\documentclass[a4paper]{amsart}
\usepackage{amssymb}
\usepackage{amsmath}
\usepackage{longtable} 
\usepackage{amsmath,amssymb,amsthm}
\usepackage{enumerate} 
\usepackage{tikz}
\usepackage{array}
\usepackage{longtable} 
\usepackage{amsmath,amssymb,amsthm}
\usepackage{enumerate}
\usepackage{array}
\newcolumntype{C}[1]{>{\centering\arraybackslash}m{#1}}
\newcolumntype{L}[1]{>{\raggedright\arraybackslash}m{#1}}

\newtheorem{theorem}{Theorem}[section]
\newtheorem{lemma}[theorem]{Lemma}
\newtheorem{proposition}[theorem]{Proposition}
\newtheorem{corollary}[theorem]{Corollary}
\theoremstyle{definition}
\newtheorem{definition}[theorem]{Definition}
\newtheorem{example}[theorem]{Example}
\theoremstyle{remark}
\newtheorem{remark}[theorem]{Remark}

\newtheorem{convention}[theorem]{Convention}
\newtheorem{notation}[theorem]{Notation}
\numberwithin{equation}{section}
\long\def\symbolfootnote[#1]#2{\begingroup%
\def\thefootnote{\fnsymbol{footnote}}\footnote[#1]{#2}\endgroup} 
\newcommand{\Ann}[1]{{#1}^0}
\newcommand{\e}{\mathrm{e}}

\newcommand{\hook}{\lrcorner \,}

\newcommand{\ad}{\mathrm{ad}}

\newcommand{\diag}{\mathrm{diag}}

\newcommand{\bZ}{\mathbb{Z}}
\newcommand{\bC}{\mathbb{C}}
\newcommand{\bR}{\mathbb{R}}

\renewcommand{\Re}{\mathrm{Re}} 
\renewcommand{\Im}{\mathrm{Im}}

\renewcommand{\dim}[1]{\mathrm{dim}(#1)}
\newcommand{\spa}[1]{\mathrm{span}(#1)}

\newcommand{\g}{\mathfrak{g}} 
\newcommand{\h}{\mathfrak{h}}
\newcommand{\uf}{\mathfrak{u}} 
 
\newcommand{\GL}{\mathrm{GL}} 
\newcommand{\SL}{\mathrm{SL}}
\newcommand{\SO}{\mathrm{SO}}

\newcommand{\SU}{\mathrm{SU}}
\newcommand{\G}{\mathrm{G}}
\newcommand{\Spin}{\mathrm{Spin}}

\newcommand{\tr}{\mathrm{tr}}

\renewcommand{\L}{\Lambda}

\begin{document}
\title{Calibrated and parallel structures on almost Abelian Lie algebras}
\author{Marco Freibert}
\address{Marco Freibert, Fachbereich Mathematik, Universit\"at Hamburg, Bundesstra{\ss}e 55, 
D-20146 Hamburg, Germany}
\email{freibert@math.uni-hamburg.de}



\begin{abstract}
In this article, we determine the seven-dimensional almost Abelian Lie algebras which admit calibrated or parallel $\G_2$-/$\G_2^*$-structures. Along the way, we show that certain well-established curvature restrictions for calibrated and parallel $\G_2$-structures are not valid in the $\G_2^*$ case. In more detail, we provide the first example of a Ricci-flat calibrated $\G_2^*$-structure on a compact manifold whose holonomy is not contained in $\G_2^*$. Moreover, we get examples of non-flat parallel $\G_2^*$-structures on almost Abelian Lie algebras $\g$. We give a full classification of these $\G_2^*$-structures if $\g$ is additionally nilpotent.
\end{abstract}

\maketitle
\section{Introduction}
A $\G_2$-structure on a seven-dimensional manifold is given by a three-form $\varphi\in \Omega^3 M$ with pointwise stabilizer isomorphic to $\G_2\subset \SO(7)$. $\varphi$ induces a Riemannian metric $g_{\varphi}$, an orientation and so a Hodge star operator $\star_{\varphi}$ on $M$. It is well-known, cf., e.g., \cite{FG}, that $\varphi$ is parallel with respect to the Levi-Civita connection of $g_{\varphi}$ if and only if $\varphi$ is closed and coclosed and that then the holonomy of $g_{\varphi}$ is contained in the exceptional holonomy group $\G_2$.

There is a similar story for the split real-form $\G_2^*\subseteq \SO_0(3,4)$ of the complex exceptional simple Lie group $\left(\G_2\right)_{\bC}$. As in the $\G_2$-case, $\G_2^*$-structures are three-forms $\varphi\in \Omega^3 M$ with pointwise stabilizer isomorphic to $\G_2^*\subseteq \SO_0(3,4)$. Here, we have an induced pseudo-Riemannian metric $g_{\varphi}$ of signature $(3,4)$ and again a Hodge star operator $\star_{\varphi}\varphi$. Moreover, $\varphi$ is again parallel if and only if it is closed and coclosed. If this is the case, the holonomy of $g_{\varphi}$ is contained in $\G_2^*$, which is one of the exceptional cases in Berger's \cite{Be} list of holonomy groups of irreducible non-symmetric simply-connected pseudo-Riemannian manifolds.

(Pseudo-)Riemannian metrics induced by parallel $\G_2$- or $\G_2^*$-structures have a very interesting curvature property: they are Ricci-flat. In the $\G_2$-case, this implies that parallel $\G_2$-structures on Riemannian homogeneous spaces are flat. This is not true for parallel $\G_2^*$-structures, as recently has been observed by Kath in \cite{K2}. Kath constructed such structures on pseudo-Riemannian symmetric spaces which have three-dimensional Abelian holonomy. In this article, we provide non-symmetric examples of non-flat parallel $\G_2^*$-structures on pseudo-Riemannian homogeneous spaces. These are defined in the left-invariant setting on seven-dimensional \emph{almost Abelian} Lie groups, i.e. Lie groups whose Lie algebra has a codimension one Abelian ideal $\uf$, and they have one- or two-dimensional Abelian holonomy. Note that Fino and Luj\'{a}n \cite{FL} also found examples of left-invariant parallel $\G_2^*$-structures on nilpotent Lie groups which even have full holonomy $\G_2^*$. More exactly, identifying left-invariant structures on Lie groups with the corresponding structures on the associated Lie algebra, we classify the seven-dimensional almost Abelian Lie algebras $\g$ admitting parallel $\G_2$-structures or parallel $\G_2^*$-structures, respectively. We observe that parallel $\G_2^*$-structures with non-degenerate codimension one Abelian ideals are automatically flat. So the examples of non-flat parallel $\G_2^*$-structures have degenerate $\uf$. We give a full classification of the parallel $\G_2^*$-structures with degenerate $\uf$ and show that their holonomy is Abelian and of dimension at most two. Moreover, we provide a more detailed description of the holonomy groups in the nilpotent case.

We also consider the weaker condition that only the $\G_2$- or $\G_2^*$-structure $\varphi\in \Omega^3 M$ itself is closed. These $\G_2$- or $\G_2^*$-structure are called \emph{calibrated} and are one of the basic classes of $\G_2$- or $\G_2^*$-structures, when one divides these structures according to their intrinsic torsion \cite{FG}. Moreover, calibrated $\G_2$-structures also have interesting curvature properties: If $M$ is compact, then both the validity of the Einstein condition \cite{CI} for $g_{\varphi}$ or the scalar-flatness of $g_{\varphi}$ \cite{Br2}, and so also the Ricci-flatness of $g_{\varphi}$, imply that $\varphi$ is parallel and that $g_{\varphi}$ has holonomy contained in $\G_2$. Moreover, in \cite{FFM}, it has been shown that left-invariant Einstein calibrated $\G_2$-structures on solvable Lie groups are flat. We show that all this is not true in the $\G_2^*$-case by giving an example of a compact nilmanifold with a calibrated $\G_2^*$-structure which is not parallel and does not have holonomy contained in $\G_2^*$. More generally, we classify the seven-dimensional almost Abelian Lie groups which admit calibrated $\G_2$-structures or calibrated $\G_2^*$-structures, respectively. The analogous classification problem of the seven-dimensional nilpotent Lie algebras admitting calibrated $\G_2$-structures
has been solved in \cite{CF}.

The methods we use to obtain the classification results are the same as the ones we used in \cite{F} to determine the seven-dimensional almost Abelian Lie algebras admitting cocalibrated $\G_2$-structures or cocalibrated $\G_2^*$-structures, respectively. We show that the existence of a calibrated or parallel structure on a seven-dimensional almost Abelian Lie algebra $\g$ is equivalent to the existence of a three-form of certain type on $\uf$ or to a pair of a three- and a four-form of certain type on $\uf$, respectively, such that the application of $\ad(f_7)|_{\uf}\in \mathfrak{gl}(\uf)$, $f_7\in \g\backslash \uf$, to these forms gives zero. Since $\ad(f_7)|_{\uf}$ determines the almost Abelian Lie algebra $\g$ fully, we get the desired classification results.

The article is organized as follows: After recalling basic facts about $\G_2$-structures and $\G_2^*$-structures in the Subsections \ref{subsec:G2vsp} and \ref{subsec:G2mflg}, we remind the reader in Subsection \ref{subsec:almostAbelian} of basic properties of almost Abelian Lie algebras $\g$. In Section \ref{sec:class}, we prove our classification results. We start in Subsection \ref{subsec:calibrated} with the calibrated case. In Subsection \ref{subsec:parallelnondeg}, we classify the seven-dimensional almost Abelian Lie algebras admitting parallel $\G_2$-structures or parallel $\G_2^*$-structures with non-degenerate codimension one Abelian ideal $\uf$, respectively. The case of parallel $\G_2^*$-structures with degenerate $\uf$ is treated in Subsection \ref{subsec:parallelG2stardeg}.
\section{$\G_2^{\epsilon}$-structures and almost Abelian Lie algebras}\label{sec:G2almostAbelian}
\subsection{$\G_2^{\epsilon}$-structures on vector spaces}\label{subsec:G2vsp}
In this subsection, we discuss $\G_2$- and $\G_2^*$-structures on a vector space level. We use a unifying language to treat both cases at once. We start by introducing these structures properly and recall afterwards some basic properties. In the end of this subsection, we have a closer look at the $k$-forms induced on a codimension one subspace. For more details on $\G_2$- and $\G_2^*$-structures and proofs of the mentioned facts, we refer the reader to \cite{Br1}, \cite{K1} and \cite{K2}.

We first recall the concept of model tensors:
\begin{definition}
Let $V$ be an $n$-dimensional vector space and tensors $(\Psi_1,\ldots,\Psi_l)\in T^{r_1,s_1} V\times \ldots\times T^{r_l,s_l} V$ and $(\psi_1,\ldots,\psi_l)\in T^{r_1,s_1} \bR^n\times \ldots\times T^{r_l,s_l} \bR^n$ be given. We say that the $l$-tuple \emph{$(\Psi_1,\ldots,\Psi_l)$ has model tensors $(\psi_1,\ldots,\psi_l)$} if there exists an isomorphism $f:V\rightarrow \bR^n$ such that $f^*\psi_i=\Psi_i$ for $i=1,\ldots,l$. In this case, we call $\left(f^{-1}(e_1),\ldots,f^{-1}(e_n)\right)$ an \emph{adapted basis}.
\end{definition}
Next, we define $\G_2$- and $\G_2^*$-structures using a unifying language:
\begin{definition}
Let $V$ be a seven-dimensional vector space and $\epsilon\in \{-1,1\}$. A three-form $\varphi\in \L^3 V^*$ is called a \emph{$\G_2^{\epsilon}$-structure (on $V$)} if
\begin{equation*}
\varphi_{\epsilon}:=-\epsilon \left(e^{127}+e^{347}\right)+e^{567}+e^{135}-e^{146}-e^{236}-e^{245}\in \L^3 \left(\bR^7\right)^*
\end{equation*}
is a model three-form for $\varphi$. Here, $e^1,\ldots,e^7$ is the canonical basis of $\left(\bR^7\right)^*$. If $\epsilon=-1$, we also call $\varphi$ a \emph{$\G_2$-structure} and if $\epsilon=1$, $\varphi$ is also called a $\G_2^*$-structure.
\end{definition}
\begin{remark}
The stabilizer group $\GL(7,\bR)_{\varphi_{-1}}$ of $\varphi_{-1}\in \L^3 \left(\bR^7\right)^*$ is $\G_2$, the simply-connected compact real form of the exceptional complex Lie group $\left(\G_2\right)_{\bC}$. Moreover, $\GL(7,\bR)_{\varphi_1}$ is the split real form $\G_2^*$ of $\left(\G_2\right)_{\bC}$ with $\pi_1(\G_2^*)=\bZ_2$. We like to mention that there is a close connection of $\G_2$- and $\G_2^*$-structures to the octonions or split-octonions, respectively, which is discussed at some length, e.g., in \cite{Br1}.
\end{remark}
\begin{notation}
To unify the treatment even more, we also set $\G_2^{-1}:=\G_2$ and $\G_2^1:=\G_2^*$. Then $\GL(7,\bR)_{\varphi_{\epsilon}}=\G_2^{\epsilon}$ for $\epsilon\in\{-1,1\}$.
\end{notation}
The inclusions $\G_2\subseteq \SO(7)$ and $\G_2^*\subseteq \SO_0(3,4)$ show that a $\G_2^{\epsilon}$-structure induces a pseudo-Euclidean metric and an orientation such that adapted bases are oriented and orthonormal. More exactly, we get:
\begin{lemma}\label{le:g2metric}
Let $V$ be a seven-dimensional vector space and $\varphi\in \L^3 V^*$ be a $\G_2^{\epsilon}$-structure. Then $\varphi$ induces a pseudo-Euclidean metric $g_{\varphi}$, a metric volume form $\phi(\varphi)$ and so also a Hodge star operator $\star_{\varphi}$ via
\begin{equation*}
g_{\varphi}(v,w)\phi(\varphi)=\frac{1}{6}\,\, v\hook \varphi\wedge w\hook \varphi\wedge \varphi
\end{equation*}
for $v,w\in V$ such that each adapted basis $(f_1,\ldots,f_7)$ is orthonormal and oriented. $g_{\varphi}$ is positive definite if $\epsilon=-1$. If $\epsilon=1$, then $g_{\varphi}$ has signature $(3,4)$ with $g(f_i,f_i)=-1$ for $i=1,2,3,4$ and $g(f_j,f_j)=1$ for $j=5,6,7$. Moreover, the Hodge dual $\star_{\varphi}\varphi$ is given by
\begin{equation*}
\star_{\varphi} \varphi=\epsilon \left(f^{1256}+f^{3456}\right)+f^{1234}-f^{2467}+f^{2357}+f^{1457}+f^{1367},
\end{equation*}
where $\left(f^1,\ldots, f^7\right)$ is the dual basis of the adapted basis $(f_1,\ldots,f_7)$.
\end{lemma}
\begin{proof}
This is proven, e.g., in \cite{CLSS}
\end{proof}
Next, we determine the model tensors of the induced $k$-forms on codimension one subspaces. To do so, we first have to define these model tensors:
\begin{definition}
Let $\epsilon \in \{-1,1\}$. Then we set
\begin{equation}
\begin{split}
\omega_{\epsilon}:=&-\epsilon \left(e^{12}+e^{34}\right)+e^{56}\in \L^2 \left(\bR^6\right)^*,\\
\rho_{\epsilon}:=&e^{135}+\epsilon \left(e^{146}+e^{236}+e^{245}\right)\in \L^3 \left(\bR^6\right)^*,\\
\rho_0:=& e^{126}-e^{135}+e^{234}\in \L^3 \left(\bR^6\right)^*,\\
\Omega_0:=& e^{1256}+e^{3456}\in \L^4 \left(\bR^6\right)^*.
\end{split}
\end{equation}
\end{definition}
In the next section, we will also need information on the stabilizer of the above $k$-forms and related structures. Therefore, we denote by $\GL(n,\bR)_{(\psi_1,\ldots,\psi_l)}$ the common stabilizer of the forms $\psi_1\in \L^{k_1} \left(\bR^n\right)^*,\ldots, \psi_l\in \L^{k_l} \left(\bR^n\right)^*$ on $\bR^n$ and by $Lie\left(\GL(n,\bR)_{(\psi_1,\ldots,\psi_l)}\right)$ the associated Lie algebra.
\begin{lemma}\label{le:stabilizers}
Let $\iota:\bC^3\rightarrow \bR^6$ be the isomorphism of real vector spaces given by $\iota(z_1,z_2,z_3):=(\Re(z_1),\Im(z_1),\Re(z_2),\Im(z_2),\Re(z_3),\Im(z_3))$. We get an induced monomorphism from $\mathfrak{gl}(3,\bC)$ to $\mathfrak{gl}(6,\bR)$, which we also denote by $\iota$. Using this notation, the following statements are true:
\begin{itemize}
\item[(i)]
$Lie\left(\GL(6,\bR)_{\rho_{-1}}\right)=\iota(\mathfrak{sl}(3,\bC))\cong \mathfrak{sl}(3,\bC)$
\item[(ii)]
$Lie\left(\GL(6,\bR)_{\rho_{1}}\right)=\left\{\left. \diag(A,B)\in \mathfrak{gl}(6,\bR)\right| A,\, B\in \mathfrak{sl}(3,\bR)\right\}\cong\mathfrak{sl}(3,\bR)\oplus  \linebreak\mathfrak{sl}(3,\bR)$.
\item[(iii)]
$Lie\left(\GL(6,\bR)_{\rho_0}\right)=\left\{\left.\left(\begin{smallmatrix} A & 0 \\
                                                        B & A-\tr(A) I_3 \end{smallmatrix}\right)\in \mathfrak{gl}(6,\bR) \right|A\in \mathfrak{gl}(3,\bR),\, B\in \mathfrak{sl}(3,\bR)\right\}$.
\item[(iv)] 
$Lie\left(\GL(6,\bR)_{\Omega_0}\right)=\left\{\left.\left(\begin{smallmatrix} A-\frac{\tr(C)}{2} I_4 & B \\
                                                       0 & C \end{smallmatrix}\right)\in \mathfrak{gl}(6,\bR) \right|A\in \mathfrak{sp}(4,\bR),\, B\in \bR^{4\times 2},\right.\linebreak
\left. \qquad\qquad \qquad \qquad\qquad  \qquad\qquad \qquad \qquad\qquad\quad\;\,  C\in \mathfrak{gl}(2,\bR)\right\}$,\\
 where $\mathfrak{sp}(4,\bR):=\left\{ A\in \bR^{4\times 4}\left| A^t \diag(J,J)+\diag(J,J) A=0\right.\right\}$  with $J:=\left(\begin{smallmatrix} 0 & 1 \\ -1 & 0 \end{smallmatrix}\right)\in \bR^{2\times 2}$.                                           
\item[(v)]
$Lie\left(\GL(6,\bR)_{\left(\rho_{-1},\frac{1}{2}\omega_{-1}^2\right)}\right)=\iota(\mathfrak{su}(3))\cong \mathfrak{su}(3)$.
\item[(vi)]
 $Lie\left(\GL(6,\bR)_{\left(\rho_{-1},\frac{1}{2}\omega_{1}^2\right)}\right)=\iota(\mathfrak{su}(1,2))\cong\mathfrak{su}(1,2)$.
\item[(vii)]
 $Lie\left(\GL(6,\bR)_{\left(\rho_{1},\frac{1}{2}\omega_{-1}^2\right)}\right)=\left\{\left.\diag\left(A,-A^t\right)\in \mathfrak{gl}(6,\bR)\right|A\in \mathfrak{sl}(3,\bR)\right\}\cong \linebreak \mathfrak{sl}(3,\bR)$.                                                      
\end{itemize}
\end{lemma}
\begin{proof}
(i) and (ii) are proven, for instance, in \cite{H}. From these results, (v), (vi) and (vii) immediately follow. (iii) is proven in \cite{V} and (iv) in \cite{F}
\end{proof}
We start with the model tensors of the induced forms on non-degenerate codimension one subspaces:
\begin{proposition}\label{pro:modeltensors}
Let $V$ be a seven-dimensional space, $\varphi\in \L^3 V^*$ be a $\G_2^{\epsilon}$-structure on $V$ and $W$ be a six-dimensional subspace of $V$.
\begin{enumerate}
\item
If $\epsilon=-1$, then $(\varphi|_W,\star_{\varphi}\varphi|_W)$ has model tensors $\left(\rho_{-1},\frac{1}{2}\omega_{-1}^2\right)$.
\item
If $\epsilon=1$ and $\uf$ has signature $(2,4)$ with respect to $g_{\varphi}$, then $(\varphi|_W,\star_{\varphi}\varphi|_W)$ has model tensors $\left(\rho_{-1},\frac{1}{2}\omega_1^2\right)$.
\item
If $\epsilon=1$ and $\uf$ has signature $(3,3)$ with respect to $g_{\varphi}$, then $(\varphi|_W,\star_{\varphi}\varphi|_W)$ has model tensors $\left(\rho_1,-\frac{1}{2}\omega_{-1}^2\right)$.
\end{enumerate}
\end{proposition}
\begin{proof}
If $\epsilon=-1$, then it is well-known, cf., e.g., \cite{Br1}, that $\GL(V)_{\varphi}$ acts transitively on the space of all lines in $V$. Hence it acts also transitively on the set of all six-dimensional subspaces of $V$. Thus, we may assume $W=\spa{f_1,\ldots,f_6}$, $(f_1,\ldots,f_7)$ being an adapted basis and directly get assertion (a).

If $\epsilon=1$, then \cite{Br1} tells us that $\GL(V)_{\varphi}$ acts transitively on the space of positive lines in $V$ and also on the space of negative lines in $V$. Thus, it acts also transitively on all six-dimensional subspaces of fixed non-degenerate signature and we may assume for (b) that $W=\spa{f_1,\ldots,f_6}$ and for (c) that $W=\spa{f_2,\ldots,f_7}$, where $(f_1,\ldots,f_7)$ is an adapted basis. The computation that the induced forms on $V$ have the claimed model tensors is straightforward.
\end{proof}
Now we have a closer look at the $k$-forms a $\G_2^*$-structure $\varphi$ induces on a degenerate codimension one subspace. We show that we can choose a special basis adapted to the codimension one subspace such that $\varphi$ has a particular form.
\begin{lemma}\label{le:Wittbasis}
Let $\varphi\in \L^3 V^*$ be a $\G_2^*$-structure on a seven-dimensional vector space $V$ and $W$ be a six-dimensional subspace of $V$ such that $W$ is degenerate with respect to the induced pseudo-Euclidean metric $g_{\varphi}$. Then there exists a basis $F_1,\ldots,F_7$ of $V$ such that $F_1,\ldots,F_6$ is a basis of $W$, $F_7\in V\backslash W$ and such that
\begin{equation*}
\varphi=-F^{156}-F^{236}+F^{245}-\frac{1}{2} F^{127}-F^{347}.
\end{equation*}
The Hodge dual $\star_{\varphi}\varphi$ of $\varphi$ is then given by
\begin{equation*}
\star_{\varphi}\varphi=F^{1256}+F^{3456}+\frac{1}{2} F^{1367}-\frac{1}{2} F^{1457}+F^{2347}
\end{equation*}
and the induced pseudo-Euclidean metric by
\begin{equation*}
g_{\varphi}=-(F^2)^2+F^1\cdot F^7+2 F^3\cdot F^6-2 F^4\cdot F^5.
\end{equation*}
In particular, $\varphi|_W$ has model tensor $\rho_0\in \L^3 \left(\bR^6\right)^*$ and $\star_{\varphi}\varphi|_W$ has model tensor $\Omega_0\in \L^4 \left(\bR^6\right)^*$. 
\end{lemma}
\begin{proof}
Let $\varphi$ be a $\G_2^*$-structure. Then there exists a basis $f_1,\ldots,f_7$ of $V$ such that
\begin{equation*}
\begin{split}
\varphi=&-f^{127}-f^{347}+f^{567}+f^{135}-f^{146}-f^{236}-f^{245},\\
\star_{\varphi} \varphi=&f^{1234}-f^{1256}-f^{3456}-f^{2467}+f^{2357}+f^{1367}+f^{1457},
\end{split}
\end{equation*}
and such that $g_{\varphi}=-\sum_{i=1}^4 (f^i)^2+\sum_{j=5}^7 (f^j)^2$. By \cite{Br1}, $\GL(V)_{\varphi}$ acts transitively on the set of all null lines in $V$ and so also on the set of all six-dimensional degenerate subspaces of $V$. Since $\spa{f_1+f_7,f_2,\ldots,f_6}$ is such a degenerate subspace, we may assume $W=\spa{f_1+f_7,f_2,\ldots,f_6}$. We define the basis $F_1,\ldots,F_7$ of $V$ via its dual basis $F^1,\ldots,F^7$ by setting
\begin{equation*}
\begin{split}
F^1:=f^1+f^7,\;\; F^2:=f^2,\;\;  F^3:=\frac{\sqrt{2}}{2}\left(f^3+f^6\right) ,\;\;  F^4:=\frac{\sqrt{2}}{2}\left(f^4+f^5\right),\\
F^5:=\frac{\sqrt{2}}{2}\left(f^4-f^5\right),\;\;  F^6:=\frac{\sqrt{2}}{2}\left(f^6-f^3\right),\;\;  F^7:=-f^1+f^7.
\end{split}
\end{equation*}
A short computation gives the claimed formulas for $\varphi$, $\star_{\varphi}\varphi$ and $g_{\varphi}$. Moreover,
\begin{equation*}
\spa{F_1,\ldots,F_6}=\ker(F^7)=\spa{f_1+f_7,f_2,\ldots,f_6}=W.
\end{equation*}
Thus, $F_1,\ldots,F_7$ is a basis of $V$ such that $F_1,\ldots,F_6$ is a basis of $W$ and $F_7\in V\backslash W$. A model tensor of $\star_{\varphi}\varphi|_W$ is obviously $\Omega_0\in \L^4 \left(\bR^6\right)^*$. Moreover,
\begin{equation*}
\varphi|_W=-F^{156}-F^{236}+F^{245}=F^{651}-F^{623}+F^{524},
\end{equation*}
and so $\varphi|_W$ has model tensor $\rho_0=e^{126}-e^{135}+e^{234}\in \L^3\left(\bR^6\right)^*$.
\end{proof}
\begin{remark}
In \cite{K2}, Kath shows that for each $\G_2^*$-structure $\varphi\in \L^3 V^*$ on a seven-dimensional vector space $V$ there exists a basis $F_1,\ldots,F_7$ of $V$ such that $\varphi=\sqrt{2}\left(F^{127}+F^{356}\right)-F^4\wedge\left(F^{15}+F^{26}-F^{37}\right)$. The induced metric is then given by $g_{\varphi}=-(F^4)^2+2\sum_{i=1}^3 F^i\cdot F^{i+4}$. Kath calls such a basis a \emph{Witt basis}.
The basis we construct in Lemma \ref{le:Wittbasis} is, up to a permutation and scaling of some basis vectors, a Witt basis in the sense of Kath. Additionally, our basis is also ``adapted'' to the codimension one subspace $W$. We abuse the notation and call in the following a basis as in Lemma \ref{le:Wittbasis} a \emph{Witt basis (with respect to $W$)}. Note that while the form of the pseudo-Euclidean metric $g_{\varphi}$ with respect to a Witt basis in our sense is not as easy as with respect to a Witt basis in the sense of Kath, the induced forms on $W$ look nicer in our basis.
\end{remark}
\subsection{$\G_2^{\epsilon}$-structures on manifolds and Lie algebras}\label{subsec:G2mflg}

For fixed $\epsilon \in \{-1,1\}$, a \emph{$\G_2^{\epsilon}$-structure} on a seven-dimensional manifold $M$ is a three-form $\varphi\in \Omega^3 M$ such that $\varphi_p\in \L^3 T_p^* M$ is a $\G_2^{\epsilon}$-structure on $T_p M$ for all $p\in M$. Since the stabilizer of $\varphi_p$ is conjugated to $\G_2^{\epsilon}\subseteq \GL(7,\bR)$, such a three-form $\varphi$ is the same as a $\G_2^{\epsilon}$-structure on $M$ in the usual sense, namely a reduction of the frame bundle $\mathcal{F}(M)$ to $\G_2^{\epsilon}$. Note that a $\G_2^{\epsilon}$-structure on $M$ induces a pseudo-Riemannian metric $g_{\varphi}$ of the corresponding signature and an orientation and so also a Hodge star operator $\star_{\varphi}\varphi$ by Lemma \ref{le:g2metric}. Thus, we also have the Levi-Civita connection $\nabla^{g_{\varphi}}$ of the metric $g_{\varphi}$.
\begin{proposition}\label{pro:parallelG2}
Let $M$ be a seven-dimensional manifold, $\epsilon\in \{-1,1\}$ and $\varphi\in \Omega^3 M$ be a $\G_2^{\epsilon}$-structure on $M$. Then $\varphi$ is a parallel $\G_2^{\epsilon}$-structure, i.e. $\nabla^{g_{\varphi}}\varphi=0$, if and only if $d\varphi=0$ and $d\star_{\varphi}\varphi=0$. In this case, the holonomy group $Hol(g_{\varphi})$ of $g_{\varphi}$ is a subgroup of $\G_2^{\epsilon}$ and $g_{\varphi}$ is Ricci-flat.
\end{proposition}
\begin{proof}
The first equivalence is proven for $\epsilon=-1$ in \cite{FG} and for both cases in \cite{Br1}. The statement that the holonomy group is then a subgroup of $\G_2^{\epsilon}$ is simply the holonomy principle. The Ricci-flatness of parallel $\G_2$-structures is shown in \cite{Bo}. The analogous statement for parallel $\G_2^*$-structures may be found, e.g., in the proof of \cite[Proposition 3.5]{K2}.
\end{proof}
In this article, we also consider the following weaker condition:
\begin{definition}
Let $M$ be a seven-dimensional manifold, $\epsilon\in \{-1,1\}$ and $\varphi\in \Omega^3 M$ be a $\G_2^{\epsilon}$-structure on $M$. Then $\varphi$ is called \emph{calibrated} if $d\varphi=0$.
\end{definition}
We are mainly interested in left-invariant $\G_2^{\epsilon}$-structures on Lie groups $\G$. These are in one-to-one correspondence to $\G_2^{\epsilon}$-structures on the associated Lie algebra $\g$ and via this correspondence, the exterior derivative on $\Omega^{\bullet}\G$ corresponds to a derivation $d:\L^{\bullet} \g^*\rightarrow \L^{\bullet+1} \g^*$ on $\L^{\bullet} \g^*$ uniquely defined by $df=0$ for $f\in \L^0 \g^*\cong \bR$ and by $(d\alpha)(X,Y)=-\alpha([X,Y])$ for $\alpha\in \g^*$, $X,Y\in \g$. Moreover, the Levi-Civita connection $\nabla^{\g_{\varphi}}$ on $\G$ corresponds to a bilinear map $\nabla^{g_{\varphi}}:\g\times \g\rightarrow \g$. Hence, we may speak of parallel and calibrated $\G_2^{\epsilon}$-structures on Lie algebras $\g$, and these structures correspond to left-invariant parallel and left-invariant calibrated $\G_2^{\epsilon}$-structures on any Lie group $\G$ with Lie algebra $\g$, respectively.

The geometry of parallel $\G_2$-structures on Lie algebras is very restrictive due to the following classical result of Alekseevsky and Kimel'fel'd \cite{AK}:
\begin{theorem}\label{th:AK}[Alekseevsky, Kimel'fel'd]
Ricci-flat Riemannian homogeneous \linebreak spaces are flat.
\end{theorem}
\begin{corollary}\label{co:parallelflatG2}
Let $\g$ be a Lie algebra and $\varphi\in \L^3 \g^*$ be a parallel $\G_2$-structure. Then the induced Riemannian metric $g_{\varphi}$ is flat.
\end{corollary}
\begin{remark}
Theorem \ref{th:AK} does not hold in the pseudo-Riemannian case, cf., e.g., \cite{AC}, \cite{K2} and \cite{KO}. Below, we also provide more examples of Ricci-flat pseudo-Riemannian homogeneous spaces which are not flat.
\end{remark}
\subsection{Almost Abelian Lie algebras}\label{subsec:almostAbelian}
In this section, we give a rough review of almost Abelian Lie algebras. More details and proofs of the results mentioned below may be found in \cite{F}.
\begin{definition}
An $n$-dimensional Lie algebra $\g$ is called \emph{almost Abelian} if $\g$ admits a codimension one Abelian ideal $\uf$. 
\end{definition}
Almost Abelian Lie algebras are of the form $\g=\uf\rtimes \bR f_n\cong \bR^{n-1}\rtimes_{\varphi} \bR$ for $f_n\in \g\backslash \uf$, $\varphi\in \mathrm{End}\left(\bR,\mathrm{End}\left(\bR^n\right)\right)$. Thus, the entire structure, and so also the differential, is determined by one endomorphism, namely $\ad(f_n)|_{\uf}$. Of course, by rescaling $\ad(f_n)|_{\uf}$ or conjugating with an element in $\GL(\uf)$, the Lie algebra doesn't change. More exactly, the following is true:
\begin{proposition}\label{pro:differential}
\begin{enumerate}
\item
Two $n$-dimensional almost Abelian Lie algebras $\g=\bR^{n-1} \rtimes_{\varphi} \bR $ and $\g'=\bR^{n-1} \rtimes_{\varphi'} \bR$, $\varphi,\,\varphi'\in \mathrm{End}\left(\bR,\mathrm{End}\left(\bR^n\right)\right)$, are isomorphic if and only if there exists $\gamma\in \bR^*$ and $F\in \GL(n-1,\bR)$ such that $\varphi'(1)=\gamma \cdot \left( F^{-1}\circ \varphi(1)\circ F\right)$.
\item
Let $\g$ be an almost Abelian Lie algebra, $\uf$ be a codimension one Abelian ideal and $f_n\in \g\backslash \uf$. Use the decomposition $\g=\uf\oplus \spa{f_n}$ to identify the annihilator $\Ann{f_n}$ of $f_n$ with $\uf^*$ and let $f^n\in \Ann{\uf}$ be the element with $f^n(f_n)=1$. Set $f:=\ad(f_n)|_{\uf}\in \mathfrak{gl}(\uf)$ and denote by $g.\psi$ the natural action of an element $g\in \mathfrak{gl}(\uf)$ on $\psi\in\L^k \uf^*$.  Then $\g^*=\uf^*\oplus \spa{f^n}$ and
\begin{equation*}
d\rho=f^n\wedge f.\rho,\qquad d\left(\rho\wedge f^n\right)=0 
\end{equation*}
for $\rho\in \L^k \uf^*$. Hence, a $k$-form $\rho\in \L^k \uf^*$ is closed if and only if $f\in Lie(\GL(\uf)_{\rho})$.
\end{enumerate}
\end{proposition}
\begin{remark}
Note that an almost Abelian Lie algebra has a unique codimension one Abelian ideal $\uf$ if and only if $\g\notin \left\{\bR^n,\mathfrak{h}_3\oplus \bR^{n-3}\right\}$, cf. the proof of \cite[Proposition 1]{F}.
\end{remark}
Therefore, it is natural to encode the admittance of a calibrated or parallel $\G_2^{\epsilon}$-structure in terms of properties of the real or complex Jordan normal form of $\ad(f_7)|_{\uf}$, as we will do in most cases in the next section.
\section{Classifications}\label{sec:class}
\subsection{Calibrated $\G_2^{\epsilon}$-structures}\label{subsec:calibrated}
In this subsection, we classify the seven-dimen\-sional almost Abelian Lie algebras $\g$ which admit calibrated $\G_2^{\epsilon}$-structures. One result is that all seven-dimensional nilpotent almost Abelian Lie algebras admit calibrated $\G_2^*$-structures. We give examples of calibrated $\G_2^*$-structure which are Ricci-flat but not parallel and do not have holonomy contained in $\G_2^*$.

We start with the main theorem of this subsection.
\begin{theorem}\label{th:calibratedG2}
Let $\g$ be a seven-dimensional real almost Abelian Lie algebra and $\uf$ be a six-dimensional Abelian ideal in $\g$.
\begin{enumerate}
\item
The following are equivalent:
\begin{enumerate}
\item
$\g$ admits a calibrated $\G_2$-structure.
\item
$\g$ admits a calibrated $\G_2^*$-structure such that $\uf$ has signature $(2,4)$ with respect to the induced pseudo-Euclidean metric on $\g$.
\item
For any $f_7\in \g\backslash \uf$, there exist $A,\, B\in \mathfrak{sl}(3,\bR)$ and an ordered basis $(f_1,\ldots,f_6)$ of $\uf$ such that the transformation matrix of $\ad(f_7)|_{\uf}$ with respect to $(f_1,\ldots,f_6)$ is given by
\begin{equation*}
\begin{pmatrix}
A & B \\
-B & A
\end{pmatrix}.
\end{equation*}
\item
For any $f_7\in \g\backslash \uf$, the complex Jordan normal form of $\ad(f_7)|_{\uf}$ is given, up to a permutation of the complex Jordan blocks, by $\diag\left(J,\overline{J}\right)$ for some trace-free matrix $J\in \bC^{3\times 3}$ in complex Jordan normal form.
\end{enumerate}
\item
The following are equivalent:
\begin{enumerate}
\item
$\g$ admits a calibrated $\G_2^*$-structure such that $\uf$ has signature $(3,3)$ with respect to the induced pseudo-Euclidean metric on $\g$.
\item
For any $f_7\in \g\backslash \uf$, there exist $A,\, B\in \mathfrak{sl}(3,\bR)$ and an ordered basis $(f_1,\ldots,f_6)$ of $\uf$ such that the transformation matrix of $\ad(f_7)|_{\uf}$ with respect to $(f_1,\ldots,f_6)$ is given by $\diag(A,B)$.
\item
For any $f_7\in \g\backslash \uf$, the complex Jordan normal form of $\ad(f_7)_{\uf}$ is given, up to a permutation of the complex Jordan blocks, by \linebreak $\diag(J_1,J_2)$ for trace-free matrices $J_1,J_2\in \bC^{3\times 3}$ which are complex Jordan normals form of real $3\times 3$-matrices.
\end{enumerate}
\item
The following are equivalent:
\begin{enumerate}
\item
$\g$ admits a calibrated $\G_2^*$-structure such that $\uf$ is degenerate with respect to the induced pseudo-Euclidean metric on $\g$.
\item
For any $f_7\in \g\backslash \uf$, there exists an ordered basis $(f_1,\ldots,f_6)$ of $\uf$, $A\in \mathfrak{gl}(3,\bR)$ and $B\in \mathfrak{sl}(3,\bR)$ such that the transformation matrix of $\ad(f_7)|_{\uf}$ with respect to $(f_1,\ldots,f_6)$ is given by
\begin{equation*}
\begin{pmatrix}
A & 0 \\
B & A-\tr(A) I_3
\end{pmatrix}
\end{equation*}
\end{enumerate}
\end{enumerate}
\end{theorem}
\begin{remark}
\begin{itemize}
\item
We like to point out that in (a), (iv), we, in fact, allow $J\in \bC^{3\times 3}$ to be any complex matrix in complex Jordan normal form, whereas in (b), (iii) we really require that $J_1,\, J_2\in \bC^{3\times 3}$ are complex Jordan normal forms of real $3\times 3$-matrices. That means in the concrete situation that $J_1$ and $J_2$ do either contain no Jordan block with a non-real number on the diagonal or exactly two Jordan blocks of size $1$ with a non-real number and its complex conjugate, respectively, on the diagonal.
\item
In principle, we may also express the equivalent conditions (i) and (ii) in (c) in terms of properties of the complex Jordan normal form of $\ad(f_7)|_{\uf}$. However, the precise statement is very complicated and not very instructive. This is why we will only investigate the nilpotent case below.
\end{itemize} 
\end{remark}
\begin{proof}[Proof of Theorem \ref{th:calibratedG2}.]
We fix $f_7\in \g\backslash \uf$, denote by $f^7$ the element in $\Ann{\uf}$ with $f^7(f_7)=1$ and identify $\Ann{f_7}$ with $\uf^*$ using the decomposition $\g=\uf\oplus \spa{f_7}$. In the following, we consider always the differential with respect to $\g$.

Let $\varphi\in \L^3 \g^*$ be a calibrated $\G_2^{\epsilon}$-structure. There are unique $\omega\in \L^2 \uf^*$, $\rho\in \L^3 \uf^*$ with $\varphi=\omega\wedge f^7+\rho$. Proposition \ref{pro:differential} (b) implies
\begin{equation*}
0=d\varphi=d(\omega\wedge f^7+\rho)=d\rho.
\end{equation*}
By Proposition \ref{pro:modeltensors}, $\rho$ has model tensor $\rho_{-1}\in \L^3 \left(\bR^6\right)^*$ if $\epsilon=1$ and $\uf$ has signature $(2,4)$ or if $\epsilon=-1$, $\rho$ has model tensor $\rho_1\in \L^3 \left(\bR^6\right)^*$ if $\epsilon=1$ and $\uf$ has signature $(3,3)$ and $\rho$ has model tensor $\rho_0\in \L^3 \left(\bR^6\right)^*$ if $\epsilon=1$ and $\uf$ is degenerate.

Conversely, let $\rho\in \L^3 \uf^*\cong \L^3 \Ann{f_7} $ be closed with model tensor $\rho_{-1}$. Choose an arbitrary $\G_2$-structure $\tilde{\varphi}\in \L^3 \g^*$ and an arbitrary $\G_2^*$-structure $\check{\varphi}\in \L^3 \g^*$ such that $\uf$ has signature $(2,4)$ with respect to the induced pseudo-Euclidean metric $g_{\check{\varphi}}$. We decompose
\begin{equation*}
\tilde{\varphi}=\tilde{\omega}\wedge f^7+\tilde{\rho},\qquad \check{\varphi}=\check{\omega}\wedge f^7+\check{\rho}
\end{equation*}
with $\tilde{\omega},\,\check{\omega}\in \L^2 \uf^*$ and $\tilde{\rho},\,\check{\rho}\in \L^3 \uf^*$. By Proposition \ref{pro:modeltensors}, both $\tilde{\rho}$ and $\check{\rho}$ have model tensor $\rho_{-1}$. Hence, there are isomorphisms $\tilde{F},\,\check{F}:\uf\rightarrow \uf$ with $\tilde{F}^*\tilde{\rho}=\rho=\check{F}^*\check{\rho}$. We define isomorphisms
$\tilde{G},\,\check{G}:\g\rightarrow \g$ by $\tilde{G}|_{\uf}:=\tilde{F}$, $\check{G}|_{\uf}:=\check{F}$ and $\tilde{G}(f_7):=f_7=:\check{G}(f_7)$. Then $\tilde{G}^*\tilde{\varphi}$ is a $\G_2$-structure with $\tilde{G}^*\tilde{\varphi}|_{\uf}=\rho$ and the closure of $\rho$ and Proposition \ref{pro:differential} (b) show that $\tilde{G}^*\tilde{\varphi}$ is closed. Moreover, by the same arguments $\check{G}^*\check{\varphi}$ is a calibrated $\G_2^*$-structure with $\check{G}^*\check{\varphi}|_{\uf}=\rho$. Since $\check{G}$ is an isometry between $(\g,g_{\check{G}^*\check{\varphi}})$ and $(\g,g_{\check{\varphi}})$, the signature of $\uf$ is $(2,4)$ with respect to $g_{\check{G}^*\check{\varphi}}$. Similarly, we see that for each closed $\rho\in \L^3 \uf^*$ with model tensor $\rho_1$ there exists a calibrated $\G_2^*$-structure $\hat{\varphi}\in \L^3 \g^*$ with $\hat{\varphi}|_{\uf}=\rho$ and $\uf$ having signature $(3,3)$ with respect to $g_{\hat{\varphi}}$. Moreover, we also can show in the same way that for each closed $\rho\in \L^3 \uf^*$ with model tensor $\rho_0$ there exists a calibrated $\G_2^*$-structure $\overline{\varphi}\in \L^3 \g^*$ with $\overline{\varphi}|_{\uf}=\rho$ and $\uf$ being degenerate.

Summarizing, the existence of a calibrated $\G_2^{\epsilon}$-structure $\varphi\in \L^3 \g^*$ such that $g_{\varphi}|_{\uf}$ has the desired property is equivalent to the existence of a closed three-form $\rho\in \L^3 \uf^*\cong \L^3 \Ann{f_7}$ with the corresponding model tensor mentioned above. Now Proposition \ref{pro:differential} (b) tells us that the closure of $\rho$ is equivalent to $\ad(f_7)|_{\uf}\in Lie(\GL(\uf)_{\rho})$. Hence, the equivalence of (i)-(iii) in (a), of (i) and (ii) in (b) and of (i) and (ii) in (c) follows 
from Lemma \ref{le:stabilizers}. The equivalence of (iii) and (iv) in (a) follows from the fact that $Lie(\GL(6,\bR)_{\rho_{-1}})=i(\mathfrak{sl}(3,\bC))$ for some injective $\bR$-Lie algebra homomorphism $i:\mathfrak{gl}(3,\bC)\rightarrow \mathfrak{gl}(6,\bC)$ and that if $J$ is a complex Jordan normal form for $A\in \mathfrak{gl}(3,\bC)$, then $\diag\left( J,\overline{J}\right)$ is a complex Jordan normal form for $i(A)$. The equivalence of (ii) and (iii) in (b) is obvious.
\end{proof}

\begin{remark}
\begin{itemize}
\item
Note that by Theorem \ref{th:calibratedG2}, a seven-dimensional almost Abelian Lie algebra admitting a calibrated $\G_2^{\epsilon}$-structure with non-degenerate codimension one Abelian ideal is necessarily unimodular.
\item
In \cite{F}, the author shows that a seven-dimensional almost Abelian Lie algebra $\g$ with codimension one Abelian ideal $\uf$ admits a cocalibrated $\G_2^*$-structure such that $\uf$ has signature $(2,4)$ if and only if $\g$ admits a cocalibrated $\G_2^*$-structure such that $\uf$ has signature $(3,3)$. Moreover, the existence of a cocalibrated $\G_2^*$-structure with non-degenerate $\uf$ implies the existence of a cocalibrated $\G_2^*$-structure with degenerate $\uf$. The corresponding relations do not hold for calibrated $\G_2^*$-structures:
\begin{itemize}
\item
If the complex Jordan normal form of $\ad(f_7)|_{\uf}$ is given by $\diag(1+i,2+2i,-3-3i,1-i,2-2i,-3+3i)$, then Theorem \ref{th:calibratedG2} shows that $\g$ admits a calibrated $\G_2^*$-structure such that $\uf$ has signature $(2,4)$ but neither one such that $\uf$ has signature $(3,3)$ nor one such that $\uf$ is degenerate.
\item
If the complex Jordan normal form of $\ad(f_7)|_{\uf}$ is given by $\diag(1,2,\linebreak -3, 4,5,-9)$, then Theorem \ref{th:calibratedG2} shows that $\g$ admits a calibrated $\G_2^*$-structure such that $\uf$ has signature $(3,3)$ but neither one such that $\uf$ has signature $(2,4)$ nor one such that $\uf$ is degenerate.
\item
If the complex Jordan normal form of $\ad(f_7)|_{\uf}$ is given by $\diag(1,2,\linebreak 3,-5,-4,-3)$, then Theorem \ref{th:calibratedG2} shows that $\g$ admits a calibrated $\G_2^*$-structure with degenerate $\uf$ but neither one where $\uf$ has signature $(2,4)$ nor one where $\uf$ has signature $(3,3)$.
\end{itemize}
\end{itemize}
\end{remark}
Next, we take a closer look at calibrated $\G_2^{\epsilon}$-structures on seven-dimensional nilpotent almost Abelian Lie algebras. By Engel's theorem, an almost Abelian Lie algebra $\g$ with codimension one Abelian ideal $\uf$ is nilpotent if and only if $\ad(f_n)|_{\uf}$ is nilpotent for $f_n\in \g\backslash \uf$. Thus, for each partition $n_1+\ldots+n_k=6$ of $6$ with $n_1,\ldots,n_k\in \{1,\ldots,6\}$, $n_1\geq \ldots \geq n_k$, there is exactly one seven-dimensional nilpotent almost Abelian Lie algebra, namely that one whose (complex or real) Jordan normal form has Jordan blocks of sizes $n_1,\ldots,n_k$. Therefore, in total we have $11$ such nilpotent Lie algebras and they are listed in Table \ref{table1} in the appendix. All of them have rational structure constants. So the simply-connected Lie group $\G$ with Lie algebra $\g$ admits a cocompact lattice $\Lambda$ and we get compact nilmanifolds $\G/\Lambda$ with calibrated $\G_2^{\epsilon}$-structures. Theorem \ref{th:calibratedG2} shows that the seven-dimensional nilpotent almost Abelian Lie algebras which admit calibrated $\G_2$-structures are given by $\bR^7,\, A_{5,1}\oplus \bR^2,\, \mathfrak{n}_{7,2}$, which is in accordance with the results obtained in \cite{CF}. In contrast, we can show that all seven-dimensional nilpotent almost Abelian Lie algebras admit calibrated $\G_2^*$-structures.
\begin{corollary}\label{co:calibratednilpotent}
Let $\g$ be a seven-dimensional nilpotent almost Abelian Lie algebra with codimension one Abelian ideal $\uf$. Then the following is true:
\begin{enumerate}
\item
$\g$ admits a calibrated $\G_2$-structure if and only if $\g\in \left\{\bR^7,\, A_{5,1}\oplus \bR^2,\, \mathfrak{n}_{7,2}\right\}$.
\item
 $\g$ admits a calibrated $\G_2^*$-structure. More exactly, $\g$ admits a calibrated $\G_2^*$-structure with non-degenerate $\uf$ if and only if
\begin{equation*}
\g\in \left\{\mathfrak{n}_{7,2},\, \mathfrak{n}_{6,1}\oplus \bR,\, A_{5,1}\oplus \bR^2,\, A_{4,1}\oplus \bR^3,\, \h_3\oplus \bR^4,\,  \bR^7 \right\},
\end{equation*}
and $\g$ admits a calibrated $\G_2^*$-structure with degenerate $\uf$ if and only if $\g\neq A_{4,1}\oplus \bR^3$.
 \end{enumerate}
\end{corollary}
\begin{proof}
Part (a) follows immediately from Theorem \ref{th:calibratedG2} as well as the classification of the nilpotent almost Abelian Lie algebras admitting calibrated $\G_2^*$-structures with non-degenerate $\uf$. To finish the proof, we have use Theorem \ref{th:calibratedG2} (c) to classify the nilpotent almost Abelian Lie algebras which admit a calibrated $\G_2^*$-structure with degenerate $\uf$.
By Table \ref{table1}, the sizes of the Jordan blocks in the Jordan normal form of $A_{4,1}\oplus \bR^3$ is $(3,1,1,1)$. Thus, by Theorem \ref{th:calibratedG2} (c), we first have to find for each partition $(n_1,\ldots,n_k)\in \left\{1,\ldots,6\right\}^k$, $n_1+\ldots+n_k=6$, $n_1\geq \ldots \geq n_k$, of $6$ which is not equal to $(3,1,1,1)$ matrices $A\in \mathfrak{gl}(3,\bR)$ and $B\in \mathfrak{sl}(3,\bR)$ such that the Jordan normal form of $\left(\begin{smallmatrix} A & 0 \\ B & A \end{smallmatrix}\right)$ consists of Jordan blocks of sizes $(n_1,\ldots,n_k)$ with zeros on the diagonal. To do so, we denote by $J_{m}$ a complex Jordan block of size $m$ with $0$ on the diagonal such that the $1$s in $J_m$ are on the superdiagonal. Moreover, we denote by $e_1,e_2,e_3$ the canonical basis of $\bR^3$ and write $(v_1,v_2,v_3)$ for the $3\times 3$-real matrix with columns $v_1,\,v_2,\,v_3\in \bR^3$. Then, the following is true:
\begin{itemize}
\item
A Jordan normal form with Jordan blocks of sizes $(1,1,1,1,1,1)$, $(2,1,1,1,\linebreak 1)$, $(2,2,1,1)$ or $(2,2,2)$ may be achieved by choosing $A=0$ and $B\in \mathfrak{sl}(3,\bR)$ of rank $0$, $1$, $2$ or $3$, respectively.
\item
A Jordan normal form with Jordan blocks of sizes $(3,2,1)$, $(4,1,1)$ or $(4,2)$ may be achieved by $A=\diag(J_2,0)$ and $B=(e_1,0,0)$, $B=(e_2,0,0)$ or $B=(e_2,-e_2,e_3)$, respectively.
\item
A Jordan normal form with Jordan blocks of sizes $(3,3)$, $(5,1)$ or $(6)$ may be achieved by choosing $A=J_3$ and $B=0$, $B=(e_2, 0, 0)$ or $B=(e_3, 0, 0)$, respectively.
\end{itemize}
What is missing is to show that there are no $A\in \mathfrak{gl}(3,\bR)$ and $B\in \mathfrak{sl}(3,\bR)$ such that the Jordan normal form of $C:=\left(\begin{smallmatrix} A & 0 \\ B & A \end{smallmatrix}\right)$ has Jordan normal form with Jordan blocks of sizes $(3,1,1,1)$. So let us assume the contrary. Then the rank of $C$ is two and so the rank of $A$ is less than two. It cannot be $0$, since then $C^2$ would be zero. Hence, it is one and we may assume that $A=\diag(J_2(0),0)$. But then $B$ has to be equal to $B=(0,v,0)$ for some $v\in \bR^3$, since otherwise the rank of $C$ would be greater than two. Since $B$ is trace-free, we have $v=(v_1,0,v_3)$ for $v_1,v_3\in \bR$. But then $C^2=0$, a contradiction. This finishes the proof.
\end{proof}
We give two examples of Ricci-flat calibrated $\G_2^*$-structures on almost Abelian Lie algebras which are not parallel and do not have holonomy contained in $\G_2^*$. The first one is defined on a nilpotent Lie algebra and gives rise to a compact nilmanifold with such a structure. To the best of the author's knowledge, this is the first example of a calibrated $\G_2^*$-structure on a compact manifold with the mentioned properties. Recall that in the $\G_2$-case there cannot be any Ricci-flat calibrated $\G_2$-structure on a compact manifold which is not parallel, cf. \cite{Br2} and \cite{CI}.
\begin{example}
\begin{enumerate}
\item
We consider the Lie algebra $\g:=\spa{f_1,\ldots,f_7}$ uniquely defined by $[f_3,f_7]:=f_1$, $[f_4,f_7]:=f_3$, $[f_5,f_7]:=f_2$, $[f_6,f_7]:=f_5$ and $[f_i,f_j]:=0$ for all other $1\leq i<j\leq 7$. Then $\g$ is nilpotent almost Abelian with codimension one Abelian ideal $\uf:=\spa{f_1,\ldots,f_6}$. Note that $\g\cong \mathfrak{n}_{7,2}$. By Lemma \ref{le:Wittbasis}, the three-form
\begin{equation*}
\varphi:=-f^{156}-f^{236}+f^{245}-\frac{1}{2} f^{127}-f^{347}\in \L^3 \g^*
\end{equation*}
is a $\G_2^*$-structure on $\g$ such that $\uf$ is degenerate with respect to the induced pseudo-Riemannian metric $g_{\varphi}$. We have $d\varphi=0$, i.e. $\varphi$ is calibrated. Moreover, we may use the explicit form of $\star_{\varphi}\varphi$ given in Lemma \ref{le:Wittbasis}, to deduce that $d\star_{\varphi}\varphi=-f^{23567}\neq 0$. A straightforward calculation, which has effectively being carried out using Maple, shows that the only non-zero curvature endomorphisms $R(f_i,f_j)$ with $1\leq i<j\leq 7$ are given by
\begin{equation*}
\begin{split}
R(f_2,f_7)=&-f^6\otimes f_1+\frac{1}{2} f^7\otimes f_3,\\
R(f_5,f_7)=&-\frac{3}{2}f^5\otimes f_1-\frac{3}{4} f^7\otimes f_4,\\
R(f_6,f_7)=&-f^2\otimes f_1-\frac{1}{2} f^7\otimes f_2.
\end{split}
\end{equation*}
Hence, $g_{\varphi}$ is Ricci-flat. Moreover, 
\begin{equation*}
\begin{split}
\nabla_{f_i}(R(f_2,f_7))=&\, \nabla_{f_i} (R(f_5,f_7))=\nabla_{f_i} (R(f_6,f_7))=0,\\
\nabla_{f_7} (R(f_2,f_7))=&\, 0,\quad \nabla_{f_7} (R(f_5,f_7))= \frac{3}{2} R(f_2,f_7),\\
 \nabla_{f_7} (R(f_6,f_7))=&\, \frac{1}{3} R(f_5,f_7)
\end{split}
\end{equation*}
for all $i=1,\ldots,6$. Since $g_{\varphi}$ is real-analytic, the Ambrose-Singer theorem shows that the holonomy algebra of $g_{\varphi}$ is given by $\spa{R(f_2,f_7),\linebreak R(f_5,f_7), R(f_6,f_7)}$. Thus, the holonomy group of $g_{\varphi}$ is three-dimensional and Abelian. Moreover,
\begin{equation*}
R(f_6,f_7).\varphi =f^{256}-\frac{1}{2}f^{367}+\frac{1}{2} f^{457} \neq 0,
\end{equation*}
and so the holonomy is not a subgroup of $\GL(\g)_{\varphi}\cong \G_2^*$.
\item
We give another example of a seven-dimensional almost Abelian Lie algebra $\g$ with calibrated non-parallel Ricci-flat $\G_2^*$-structure $\varphi\in \L^3 \g^*$ whose holonomy is not contained in $\G_2^*$. Let $\g:=\spa{f_1,\ldots,f_7}$ uniquely defined by $[f_i,f_7]=-2 f_i$ for $i=1,3,4$ and $[f_j,f_7]=f_j$ for $j=2,5,6$ and $[f_k,f_l]=0$ for all other $1\leq k<l\leq 7$. Then $\uf=\spa{f_1,\ldots,f_6}$ is a codimension one Abelian ideal in $\g$. We consider again the three-form
\begin{equation*}
\varphi:=-f^{156}-f^{236}+f^{245}-\frac{1}{2} f^{127}-f^{347}\in \L^3 \g^*,
\end{equation*}
from which we know by Lemma \ref{le:Wittbasis} that it is a $\G_2^*$-structure on $\g$ with degenerate $\uf$. A short computation gives $d\varphi=0$ and $d\star_{\varphi}\varphi=f^{12567}-f^{34567}\neq 0$, where we again use that we have an explicit formula for $\star_{\varphi}\varphi$ by Lemma \ref{le:Wittbasis}. Moreover, a nasty but straightforward calculation yields the identity
\begin{equation*}
\begin{split}
\mathcal{R}:=& \spa{R(X,Y)|X,Y\in \g}\\
=& \mathrm{span}\left(2 f^2\otimes f_1+f^7\otimes f_2,2 f^6\otimes f_1-f^7\otimes f_3,\right.\\
&\left. 2 f^5\otimes f_1+f^7\otimes f_4, 2 f^4\otimes f_1+f^7\otimes f_5, 2f^3\otimes f_1-f^7\otimes f_6\right).
\end{split}
\end{equation*}
and that $g_{\varphi}$ is Ricci-flat. Moreover, $\nabla_{f_i} (R(f_j,f_k))\in \mathcal{R}$ for all $i,j,k\in \linebreak\{1,\ldots,7\}$. Hence, the holonomy algebra equals $\mathcal{R}$ and so the holonomy group of $g_{\varphi}$ is five-dimensional and Abelian. Since $\left(2 f^2\otimes f_1+f^7\otimes f_2\right). \varphi\neq 0$, as we computed in (a), the holonomy group of $g_{\varphi}$ is not contained in $\GL(\g)_{\varphi}\cong \G_2^*$. Note that $\g$ is not unimodular and so cannot admit a cocompact lattice.
\end{enumerate}
\end{example}
We would like to point out that the methods we use here and in \cite{F} to obtain the classification results may also be used to classify the almost Abelian Lie algebras admitting other types of $\G$-structures. As an example, we may consider so-called \emph{symplectic half-flat $\SU(3)$-structures} on six-dimensional almost Abelian Lie algebras $\h$. These are given by pairs $(\omega,\rho)\in \L^2 \h^*\times \L^3 \h^*$ with model tensor $(\omega_{-1},\rho_{-1})\in \L^2 \left(\bR^6\right)^*\times \L^3 \left(\bR^6\right)^*$ such that $d\omega=0$ and $d\rho=0$. Looking again at the model tensors of the induced two- and three-form on a codimension one Abelian ideal $\uf$, we see that $\h$ admits such a structure if and only if for $f_6\in \h\backslash \uf$ the endomorphism $\ad(f_6)|_{\uf}$ of $\uf$ is in a Lie subalgebra of $\mathfrak{gl}(\uf)$ conjugate to $\mathfrak{sl}(2,\bC)\subseteq \mathfrak{sl}(4,\bR)\subseteq \mathfrak{gl}(5,\bR)$. In particular, $\h=\mathfrak{k}\oplus \bR$ for some five-dimensional almost Abelian Lie algebra $\h$. Note that this case has already been treated by different methods in \cite{FMOU} and our result here is in accordance with the results there.
\subsection{Parallel $\G_2^{\epsilon}$-structures with non-degenerate $\uf$}\label{subsec:parallelnondeg}
In this subsection, we determine the seven-dimensional almost Abelian Lie algebras $\g$ admitting parallel $\G_2$-structures or parallel $\G_2^*$-structures with non-degenerate $\uf$, respectively. The case of parallel $\G_2^*$-structures with degenerate $\uf$ will be treated in the next subsection.

We start with the theorem we want to prove. In contrast to the calibrated case, we prefer in some cases to give directly all possible {\it real} Jordan normal forms.
\begin{theorem}\label{th:parallelG2}
Let $\g$ be a seven-dimensional real almost Abelian Lie algebra with six-dimensional Abelian ideal $\uf$. We set $M_{a,b}:=\left(\begin{smallmatrix} a & b \\ -b & a \end{smallmatrix}\right)$ for $a,\,b\in \bR$ and let $f_7\in \g\backslash \uf$. Then:
\begin{enumerate}
\item
$\g$ admits a parallel $\G_2$-structure if and only if there exists $a,\,b\in \bR$ and a basis $(f_1,\ldots,f_6)$ of $\uf$ such that $\ad(f_7)|_{\uf}=\diag(M_{0,a}, M_{0,b}, M_{0,-a-b})$ with respect to $(f_1,\ldots,f_6)$.
\item
$\g$ admits a parallel $\G_2^*$-structure such that $\uf$ has signature $(2,4)$ if and only if there exists a basis $(f_1,\ldots,f_6)$ of $\uf$ such that
\begin{equation*}
\begin{split}
& \ad(f_7)|_{\uf}\in \left\{\left(\begin{smallmatrix}
M_{a,b} &  & \\
& M_{-a,b} & \\
& & M_{0,-2b}
\end{smallmatrix}\right),
\left(\begin{smallmatrix}
M_{0,c} &  & \\
& M_{0,d} & \\
& & M_{0,-(c+d)}
\end{smallmatrix}\right),\right. \\
& \left.\left. \qquad\qquad\quad\, \left(\begin{smallmatrix}
M_{0,e} & I_2 & \\
& M_{0,e} & \\
& & M_{0,-2e}
\end{smallmatrix}\right),  \left(\begin{smallmatrix}
 0 & & I_2 & & \\
& & 0 & &  I_2 \\
& & & & 0
\end{smallmatrix}\right) \right| a\in \bR^*,\,\, b,\,c,\,d,\, e\in \bR\right\}
\end{split}
\end{equation*}
with respect to $(f_1,\ldots,f_6)$.
\item
$\g$ admits a parallel $\G_2^*$-structure such that $\uf$ has signature $(3,3)$ if and only if the complex Jordan normal form of $\ad(f_7)|_{\uf}$ is , up to a permutation of the Jordan blocks, of the form $\diag(J_1,J_2)$ with $J_1\in  \bR^{3\times 3}$ being an arbitrary complex Jordan normal form of a real $3\times 3$-matrix with $\tr(J_1)=0$ and $J_2\in \bR^{3\times 3}$ being the complex Jordan normal form obtained from $J_1$ by multiplying the diagonal elements with $-1$.
\item
Parallel $\G_2$-structures and parallel $\G_2^*$-structure with non-degenerate $\uf$ are flat.
\end{enumerate}
\end{theorem}
\begin{proof}
The proof is completely analogous to the determination of the Lie algebras admitting calibrated structures in the previous subsection.

Let $\varphi\in \L^3 \g^*$ be a parallel $\G_2^{\epsilon}$-structure with $\uf$ being non-degenerate. Note that we may assume that $f_7\in \g\backslash \uf$ is orthogonal to $\uf$ with respect to $g_{\varphi}$ since for any other $f_7'\in \g\backslash \uf$ we have $\ad(f_7')|_{\uf}\in \bR^*\cdot \ad(f_7)|_{\uf}$. Let $f^7\in \Ann{\uf}$ be such that $f^7(f_7)=1$ and identify the annihilator $\Ann{f_7}$ of $f_7$ with $\uf^*$ using the decomposition $\g=\uf\oplus \spa{f_7}$. There are unique $\omega\in \L^2 \uf^*$, $\rho,\,\nu\in  \L^3 \uf^*$ and $\Omega\in \L^4 \uf^*$ such that
\begin{equation*}
\varphi=\omega\wedge f^7+\rho,\quad \star_{\varphi}\varphi=\nu\wedge f^7+\Omega.
\end{equation*}
By Proposition \ref{pro:parallelG2}, $d\varphi=0$ and $d\star_{\varphi}\varphi=0$. Thus, Proposition \ref{pro:differential} (b) gives us $d\rho=0$ and $d\Omega=$ and so that $f:=\ad(f_7)\in Lie\left(\GL(\uf)_{(\rho,\Omega)}\right)$.
Proposition \ref{pro:modeltensors} tells us the model tensors for $(\rho,\Omega)$ for each value of $\epsilon$ and each value of the signature of $\uf$. Moreover, Lemma \ref{le:stabilizers} gives us the Lie algebra of the stabilizer subgroup of the corresponding model tensors. We obtain that $f$ is in a subalgebra of $\mathfrak{gl}(\uf)$ conjugate to $\iota(\mathfrak{su}(3))\subseteq \mathfrak{gl}(6,\bR)$ if $\epsilon=-1$, that $f$ is in a subalgebra of $\mathfrak{gl}(\uf)$ conjugate to $\iota(\mathfrak{su}(1,2))\subseteq \mathfrak{gl}(6,\bR)$ if $\epsilon=1$ and the signature of $\uf$ is $(2,4)$ and that $f$ is in a subalgebra of $\mathfrak{gl}(\uf)$ conjugate to $\left\{\left.\diag\left(A,-A^t\right)\in \mathfrak{gl}(6,\bR)\right|A\in \mathfrak{sl}(3,\bR)\right\} \subseteq \mathfrak{gl}(6,\bR)$ if $\epsilon=1$ and the signature of $\uf$ is equal to $(3,3)$. Here, $\iota:\mathfrak{gl}(3,\bC)\rightarrow \mathfrak{gl}(6,\bR)$ is the injective $\bR$-Lie algebra homomorphism defined in Lemma \ref{le:stabilizers}.

In all three cases, we have an orthogonal decomposition $\g=\uf\oplus \spa{f_7}$ into an Abelian ideal $\uf$ of $\g$ and an Abelian subalgebra $\spa{f_7}$ which acts skew-symmetric on the Abelian ideal $\uf$. In the Riemannian case, i.e. for $\epsilon=-1$, it is well-known that this class is exactly the class of all flat Riemannian metrics on Lie algebras, cf. \cite{M}. Of course, the flatness of a parallel $\G_2$-structure has already been known from Proposition \ref{pro:parallelG2}. In the pseudo-Riemannian case, the same calculations show that the analogous class of metrics on Lie algebras consists also solely of flat metrics, although it does not exhaust the class of flat pseudo-Riemannian metrics on Lie algebras, cf., e.g., \cite{N}. This proves (d).

We come back to the proof of (a)-(c) and assume now that $f:=\ad(f_7)|_{\uf}$, $f_7\in \g\backslash \uf$ is contained in a subalgebra $\h$ of $\mathfrak{gl}(\uf)$ which is conjugate to $\iota(\mathfrak{su}(3))$, $\iota(\mathfrak{su}(1,2))$ or \linebreak $\left\{\left.\diag\left(A,-A^t\right)\in \mathfrak{gl}(6,\bR)\right|\linebreak A\in \mathfrak{sl}(3,\bR)\right\}$, respectively. By Lemma \ref{le:stabilizers}, there exists a pair $(\rho,\Omega)\in \L^3 \uf^*\times \L^4 \uf^*\cong \L^3 \Ann{f_7}\times \L^4 \Ann{f_7}$ such that $\h=Lie\left(\GL(\uf)_{(\rho,\Omega)}\right)$ and such that $(\rho,\Omega)$ has model tensors $\left(\rho_{-1},\frac{\omega_{-1}^2}{2}\right)$, $\left(\rho_{-1},\frac{\omega_{1}^2}{2}\right)$ or  $\left(\rho_{1},-\frac{\omega_{-1}^2}{2}\right)$, respectively. By Proposition \ref{pro:differential} (b), $d\rho=0$ and $d\Omega=0$.  Now take a three-form $\tilde\varphi\in \L^3 \g^*$ such that $\tilde\varphi$ is a $\G_2$-structure, a $\G_2^*$-structure such that $\uf$ has signature $(2,4)$ or a $\G_2^*$-structure such that $\uf$ has signature $(3,3)$, respectively. We decompose
\begin{equation*}
\tilde\varphi=\tilde\omega\wedge f^7+\tilde\rho,\quad \star_{\tilde\varphi}\tilde\varphi=\tilde\nu\wedge f^7+\tilde\Omega.
\end{equation*}
with $\tilde\omega\in \L^2 \uf$, $\tilde\rho,\, \tilde\nu\in \L^3 \uf^*$ and $\tilde\Omega\in \L^4 \uf^*$. By Proposition \ref{pro:parallelG2}, $(\tilde\rho,\tilde\Omega)$ and $(\rho,\Omega)$ have the same model tensors and so there exists a linear automorphism $F:\uf\rightarrow \uf$ such that $(F^*\tilde\rho,F^*\tilde\Omega)=(\rho,\Omega)$. Define the linear automorphism $G:\g\rightarrow \g$ by $G|_{\uf}:=F$ and $G(f_7):=f_7$. Then
\begin{equation*}
\varphi:=\G^*\tilde\varphi=G^*\tilde\omega\wedge f^7+\rho
\end{equation*}
is a $\G_2$-structure, a $\G_2^*$-structure such that $\uf$ has signature $(2,4)$ or a $\G_2^*$-structure such that $\uf$ has signature $(3,3)$, respectively. Moreover,
\begin{equation*}
\star_{\varphi}\varphi=\star_{G^*\tilde\varphi}G^*\tilde\varphi=G^*\star_{\tilde\varphi}\tilde\varphi=G^*\tilde\nu\wedge f^7+\Omega.
\end{equation*}
By Proposition \ref{pro:differential} (b), $d\varphi=0$ and $d\star_{\varphi}\varphi=0$ and so Proposition \ref{pro:parallelG2} shows that $\varphi$ is a parallel $\G_2$-structure, a parallel $\G_2^*$-structure such that $\uf$ has signature $(2,4)$ or a parallel $\G_2^*$-structure such that $\uf$ has signature $(3,3)$, respectively.

Now (a) follows by observing that all the complex Jordan normal forms of elements in $\iota(\mathfrak{su}(3))$ are given by $(ia,-ia,ib,-ib,-i(a+b),i(a+b))$ for certain $a,\, b\in \bR$. To deduce (b), note that in \cite{DPWZ}, all the complex Jordan normal forms of elements in $\mathfrak{u}(1,2)\subseteq \mathfrak{gl}(3,\bC)$ are determined. To get all the complex Jordan normal forms of elements in $\mathfrak{su}(1,2)\subseteq \mathfrak{gl}(3,\bC)$, we only have to require additionally that they are trace-free. Hence, the possible complex Jordan normal forms of elements in 
$\mathfrak{su}(1,2)\subseteq \mathfrak{gl}(3,\bC)$ are
\begin{equation*}
\diag(a+ib,-a+ib,-2ib),\; \diag(ic,id,-i(c+d)),\; \diag(J_2(ie),-2ie),\; J_3(0)
\end{equation*}
for $a\in \bR^*$ and $b,\,c,\,d,\,e\in \bR$, where $J_m(\lambda)$ denotes a Jordan block of size $m$ with $\lambda\in \bC$ on the diagonal. This gives us the claimed real Jordan normal forms for elements in $\iota(\mathfrak{su}(1,2))$. (c) is obvious.
\end{proof}
\begin{remark}
There are seven-dimensional almost Abelian Lie algebras which admit both a calibrated and a cocalibrated $\G_2$-structures but no parallel $\G_2$-structure. An example is provided by the nilpotent Lie algebra $\mathfrak{n}_{7,2}$.
\end{remark}
We look again at the nilpotent case. By \cite[Theorem 2.4]{M}, a nilpotent Lie algebra $\g$ admits a flat Riemannian metric if and only if $\g$ is Abelian and so Proposition \ref{co:parallelflatG2} shows that a nilpotent Lie algebra $\g$ admits a parallel $\G_2$-structure if and only if $\g$ is Abelian. This is in accordance with Theorem \ref{th:parallelG2}. For the $\G_2^*$-case with non-degenerate $\uf$ we get from Theorem \ref{th:parallelG2}:
\begin{corollary}\label{co:parallelG2}
Let $\g$ be a seven-dimensional real nilpotent almost Abelian Lie algebra and let $\uf$ be a six-dimensional Abelian ideal in $\g$. Then $\g$ admits a parallel $\G_2^*$-structure with non-degenerate $\uf$ if and only if $\g\in \left\{\bR^7, A_{5,1}\oplus \bR^2,\mathfrak{n}_{7,2}\right\}$. If this is the case, then $\g$ admits both a parallel $\G_2^*$-structure with $\uf$ having signature $(2,4)$ and a parallel $\G_2^*$-structure with $\uf$ having signature $(3,3)$.
\end{corollary}
\subsection{Parallel $\G_2^*$-structures with degenerate $\uf$}\label{subsec:parallelG2stardeg}
In this subsection, we determine all parallel $\G_2^*$-structures $\varphi\in \L^3 \g^*$ on almost Abelian seven-dimensional Lie algebras $\g$ which admit a degenerate codimension one Abelian ideal $\uf$ with respect to the induced pseudo-Riemannian metric $g_{\varphi}$. In contrast to the case of non-degenerate $\uf$, we obtain parallel $\G_2^*$-structures $\varphi$ for which $g_{\varphi}$ is a Ricci-flat, non-flat pseudo-Riemannian metric with Abelian holonomy group of dimension at most two. In the nilpotent case, we give a detailed description of the dimension of the holonomy group of $g_{\varphi}$ and identify the locally symmetric $g_{\varphi}$.
\begin{convention}
In this subsection, the \emph{holonomy group of a pseudo-Euclidean metric on a Lie algebra $\g$} should be the holonomy group of the corresponding left-invariant pseudo-Riemannian metric on the associated connected simply-connected Lie group $\tilde{\G}$ with Lie algebra $\g$ or, equivalently, the restricted holonomy group of the corresponding left-invariant pseudo-Riemannian metric on any connected Lie group $\G$ with Lie algebra $\g$.
\end{convention}
We start with a full description of all such parallel $\G_2^*$-structures.
\begin{theorem}\label{th:parallelG2*}
Let $\g$ be a seven-dimensional almost Abelian Lie algebra, $\uf$ be a six-dimen\-sional Abelian ideal in $\g$ and $\varphi\in \L^3 \g^*$ be a $\G_2^*$-structure such that $\uf$ is degenerate with respect to the induced pseudo-Euclidean metric $g_{\varphi}$. Then $\varphi$ is parallel if and only if for some Witt basis $(f_1,\ldots,f_7)$ of $\g$ there exist $A,\, B\in \bR^{2\times 2}$ and $v,\,w\in \bR^2$ such that
\begin{equation}\label{eq:structureconst}
\ad(f_7)|_{\uf}=\begin{pmatrix} -\tr(A) & -\tr(B) & v^t & w^t\\
                               0 & 0 & 0 & v^t \\
                                0 &  Jv & A-\tr(A) I_2 & B \\
                                0  &  0 & 0             & A
                                \end{pmatrix}
\end{equation}
with respect to the basis $(f_1,\ldots,f_6)$, where $J:=\begin{pmatrix} 0 & 1 \\ -1 & 0 \end{pmatrix}\in \bR^{2\times 2}$. If $\varphi$ is parallel, then the holonomy of $g_{\varphi}$ is Abelian and at most two-dimensional.
\end{theorem}
\begin{proof}
We choose a Witt basis $(f_1,\ldots,f_7)$ for $\varphi$ and set $f:=\ad(f_7)$. By Lemma \ref{le:Wittbasis}, we have
\begin{equation*}
\varphi=  -f^{156}-f^{236}+f^{245}-\frac{1}{2} f^{127}-f^{347},\,\, \star_{\varphi} \varphi= f^{1256}+f^{3456}+\frac{1}{2} f^{1367}-\frac{1}{2} f^{1457}+f^{2347}.
\end{equation*}
By Proposition \ref{pro:parallelG2}, $\varphi$ is parallel if and only if $d\varphi=0$ and $d\star_{\varphi}\varphi=0$. Thus, Proposition \ref{pro:differential} (b) gives us that $\varphi$ is parallel if and only if
\begin{equation*}
f\in Lie\left(\GL(\uf)_{(-f^{156}-f^{236}+f^{245},f^{1256}+f^{3456})}\right).
\end{equation*}
By Lemma \ref{le:Wittbasis}, $-f^{156}-f^{236}+f^{245}=f^{651}-f^{623}+f^{524}$ has model tensor $\rho_0\in \L^3 \left(\bR^6\right)^*$ and Lemma \ref{le:stabilizers} gives us that $f\in Lie\left(\GL(\uf)_{-f^{156}-f^{236}+f^{245}}\right)$ if and only if
\begin{equation}\label{eq:firststab}
f=
\begin{pmatrix}
-\tr(A) & -\tr(B)  & v^t & w^t \\
0 & \alpha & 0 & v^t \\
z & u & A-(\alpha+\tr(A)) I_2 & B \\
0 & z & 0 & A
\end{pmatrix}
\end{equation}
for certain $A,\, B\in \bR^{2\times 2}, u,\,v,\, w,\, z \in\bR^2,\,\, \alpha\in \bR$. By Lemma \ref{le:stabilizers}, 
\begin{equation*}
Lie\left(\GL(\uf)_{f^{1256}+f^{3456}}\right)=\left\{\left.\left(
\begin{smallmatrix}
C-\frac{\tr(E)}{2} I_4 & D \\
0 & E
\end{smallmatrix}\right)\right|C\in \mathfrak{sp}(4,\bR),\, D\in \bR^{4\times 2},\, E\in \bR^{2\times 2}\right\},
\end{equation*}
with $\mathfrak{sp}(4,\bR)=\left\{ C\in \bR^{4\times 4}\left| C^t \diag(J,J)+\diag(J,J) C=0\right.\right\}$. Thus, $f$ as in Equation (\ref{eq:firststab}) is in $Lie\left(\GL(\uf)_{f^{1256}+f^{3456}}\right)$ if and only if $z=0$ and
\begin{equation*}
\begin{split}
0=&\begin{pmatrix}
-\frac{\tr(A)}{2} & 0  & 0\\
-\tr(B) & \alpha+\frac{\tr(A)}{2} & u^t \\
v & 0 & A^t-\left(\alpha+\frac{\tr(A)}{2}\right) I_2 \\
\end{pmatrix}\cdot \begin{pmatrix} 0 & 1 & 0 \\
                                -1 & 0 & 0 \\
                                 0 & 0 & J
                                 \end{pmatrix}\\
&+\begin{pmatrix} 0 & 1 & 0 \\
                                -1 & 0 & 0 \\
                                 0 & 0 & J
                                 \end{pmatrix}\cdot \begin{pmatrix}
-\frac{\tr(A)}{2} & -\tr(B)  & v^t \\
0 & \alpha+\frac{\tr(A)}{2} & 0 \\
0 & u & A-\left(\alpha+\frac{\tr(A)}{2}\right) I_2 \\
\end{pmatrix}\\
=& \begin{pmatrix}
0 & \alpha  & 0 \\
-\alpha & 0 & u^t J-v^t \\
0 & Ju+v & JA+A^t J -(2\alpha+\tr(A)) J \\
\end{pmatrix}
\end{split}
\end{equation*}
Since $JA+A^t J= \tr(A) J$ for $A\in \bR^{2\times 2}$ and $J^2=-I_2$, this is equivalent to $\alpha=0$ and $u=Jv$. Hence, $\varphi$ is parallel if and only if 
\begin{equation*}
ad(f_7)|_{\uf}=f=
\begin{pmatrix}
-\tr(A) & -\tr(B)  & v^t & w^t \\
0 & 0 & 0 & v^t \\
0 & Jv & A-\tr(A) I_2 & B \\
0 & 0 & 0 & A
\end{pmatrix}
\end{equation*}
for certain $A,\, B\in \bR^{2\times 2},\,\, v,\, w\in\bR^2$.

A lengthly but straightforward calculation, which may be carried out efficiently using a computer algebra system like Maple, shows that the space $\mathcal{R}:=\linebreak \spa{\left\{R(X,Y)| X,Y\in \g\right\}}$ spanned by all curvature endomorphisms is a subspace of $V:=\spa{2f^5\otimes f_1+f^7\otimes f_4,2 f^6\otimes f_1-f^7\otimes f_3}$. Another computation shows
\begin{equation*}
\nabla_{f_i} \left(2f^5\otimes f_1+f^7\otimes f_4\right)=0,\quad \nabla_{f_i} \left(2 f^6\otimes f_1-  f^7\otimes f_3\right)=0,
\end{equation*}
for all $i=1,\ldots,6$ and
\begin{equation}\label{eq:nablaRbasis}
\begin{split}
\nabla_{f_7}\left(2 f^5\otimes f_1+f^7\otimes f_4 \right) =& (2 a_{11}+a_{22})\left(2 f^5\otimes f_1+f^7\otimes f_4 \right)\\
& +a_{12}\cdot \left(2 f^6\otimes f_1-  f^7\otimes f_3\right),\\
\nabla_{f_7}\left(2 f^6\otimes f_1-  f^7\otimes f_3\right)=& a_{21} \left(2f^5\otimes f_1+f^7\otimes f_4 \right) \\
                                                           & +(a_{11}+2 a_{22})\left(2 f^6\otimes f_1-  f^7\otimes f_3\right).
\end{split}
\end{equation}
Since $g_{\varphi}$ is real-analytic as a left-invariant metric, the Ambrose-Singer Theorem shows that the holonomy algebra $\mathfrak{hol}(g_{\varphi})$ is a subspace of $V$. Moreover, $V$ consists of commuting endomorphisms and is two-dimensional and so the holonomy is Abelian and at most two-dimensional.
\end{proof}
Note that we do not claim that all the parallel $\G_2^*$-structures obtained in Theorem \ref{th:parallelG2*} are non-isomorphic. In fact, we will use now a Lie algebra automorphism to get a more compact description of the possible endomorphisms $\ad(f_7)|_{\uf}$ in the nilpotent case. Using this description, we are able to compute explicitly the dimension of the holonomy group of the induced pseudo-Riemannian metric in dependence of the remaining parameters in this case. Moreover, we are able to identify the locally symmetric pseudo-Riemannian metrics.
\begin{theorem}\label{th:nilpotent}
Let $\g$ be a seven-dimensional nilpotent almost Abelian Lie algebra, $\uf$ be a codimension one Abelian ideal and $\varphi\in \Lambda^3 \g^*$. Then $\varphi$ is a parallel $\G_2^*$-structure with degenerate $\uf$ if and only if there exists a Witt basis $f_1,\ldots,f_7$ such that
\begin{equation}\label{eq:degnilpotent}
\ad(f_7)|_{\uf}=
\begin{pmatrix} 0 & -\tr(B) & v^t & w^t\\
                               0 & 0 & 0 & v^t \\
                                0 &  Jv & N & B \\
                                0  &  0 & 0 & N
                                \end{pmatrix}
 \end{equation}
with $N=\begin{pmatrix} 0 & \delta \\ 0 & 0 \end{pmatrix}$ for some $\delta\in \{-1,0,1\}$ and arbitrary $B=(b_{ij})_{ij}\in \bR^{2\times 2}$, $v,w\in \bR^2$. In this case, $\delta=0$ implies that $g_{\varphi}$ is flat. If $\delta\neq 0$, then the following is true:
\begin{itemize}
\item[(i)]
$g_{\varphi}$ has two-dimensional Abelian holonomy if and only if $b_{21}\neq 0$.
\item[(ii)]
$g_{\varphi}$ has one-dimensional Abelian holonomy if and only if $b_{21}=0$ and $b_{11}\neq b_{22}$.
\item[(iii)]
$g_{\varphi}$ is flat if and only if $b_{21}=0$ and $b_{11}=b_{22}$.
\item[(iv)]
$g_{\varphi}$ is locally symmetric if and only if $b_{21}=0$. In particular, if $g_{\varphi}$ is locally symmetric, the holonomy is at most one-dimensional.
\end{itemize}
\end{theorem}
\begin{proof}
By Theorem \ref{th:parallelG2*} and Engel's Theorem, $\varphi$ is parallel if and only if there exists a Witt basis $f_1,\ldots,f_7$ such that
\begin{equation*}
\ad(f_7)|_{\uf}=\begin{pmatrix} -\tr(A) & -\tr(B) & v^t & w^t\\
                               0 & 0 & 0 & v^t \\
                                0 &  Jv & A-\tr(A) I_2 & B \\
                                0  &  0 & 0             & A
                                \end{pmatrix}
\end{equation*}
and such that $\ad(f_7)|_{\uf}$ is nilpotent. If $\ad(f_7)|_{\uf}$ is nilpotent, then, necessarily, $A$ has to be nilpotent. Computing $\left(\ad(f_7)|_{\uf}\right)^6$ for a nilpotent $A$, we see that the nilpotency of $A$ is also a sufficient condition for $\ad(f_7)|_{\uf}$ being nilpotent. By Lemma \ref{le:Wittbasis}, $\varphi=-f^{156}-f^{236}+f^{245}-\frac{1}{2} f^{127}-f^{347}$. For $C\in \SL(2,\bR)$, we have $JC=C^{-t} J$ and so
\begin{equation*}
\begin{pmatrix} C & 0 \\ 0 & C \end{pmatrix}^t \cdot \begin{pmatrix} 0 & -J \\ J & 0 \end{pmatrix}\cdot \begin{pmatrix} C & 0 \\ 0 & C \end{pmatrix}= \begin{pmatrix} 0 & -C^t J C \\ C^t J C & 0 \end{pmatrix}=\begin{pmatrix} 0 & -J \\ J & 0 \end{pmatrix}.
\end{equation*}
Thus, $\diag(1,1,C,C,1)\in \GL(7,\bR)$ stabilizes $\varphi$ for all $C\in \SL(2,\bR)$. Under the corresponding change of basis, the submatrix $A$ of $\ad(f_7)_{\uf}$ gets mapped to $C^{-1} A C$. It is a well-known fact that each non-zero nilpotent $A\in \bR^{2\times 2}$ is similar to the matrix $\left(\begin{smallmatrix} 0 & 1 \\ 0 & 0 \end{smallmatrix}\right)$ and obviously this is also true if we restrict the similarity transformations to elements in $\GL(2,\bR)$ such that the absolute value of the determinant is equal to $1$. Now each such matrix can be written as a product of an element in $\SL(2,\bR)$ and $\diag(1,-1)$. Thus, we may assume that $C^{-1} A C=\left(\begin{smallmatrix} 0 & \epsilon \\ 0 & 0 \end{smallmatrix}\right)$ for some $\epsilon \in \{-1,1\}$ if $A\neq 0$ and the first claim follows.

If $\ad(f_7)|_{\uf}$ is as in the assertion, a straightforward calculation, which may be carried out efficiently using Maple, shows that the only non-zero curvature endomorphisms $R(f_i,f_j)$ with $i<j$ are $R(f_5,f_7)$ and $R(f_6,f_7)$ and that they are concretely given by
\begin{equation*}
\begin{split}
R(f_5,f_7)=&-\delta b_{21} \left(2f^6\otimes f_1-f^7\otimes f_3\right),\\
R(f_6,f_7)=&-\delta b_{21} \left(2f^5\otimes f_1+f^7\otimes f_4\right)+\delta (b_{11}-b_{22}) \left(2f^6\otimes f_1-f^7\otimes f_3\right).
\end{split}
\end{equation*}
So $g_{\varphi}$ is flat if $\delta=0$ and we may assume for the rest of the proof that $\delta \neq 0$. Since the holonomy is at most two-dimensional, it is two-dimensional if $b_{21}\neq 0$. If $b_{21}=0$, then only $R(f_6,f_7)=\delta (b_{11}-b_{22}) \left(2f^6\otimes f_1-f^7\otimes f_3\right)$ is non-zero and Equation (\ref{eq:nablaRbasis}) shows that $\nabla (R(f_6,f_7))=0$. Thus, the holonomy is one-dimensional if and only if $b_{21}=0$ and $b_{11}\neq b_{22}$ and $g_{\varphi}$ is flat if and only if $b_{21}=0$ and $b_{11}=b_{22}$. Moreover, the only non-zero component $(\nabla_{f_i} R)(f_j,f_k)$ with $j<k$ is
\begin{equation*}
\left(\nabla_{f_7} R\right)(f_6,f_7)= -2 \delta b_{21} \left(2f^6\otimes f_1-f^7\otimes f_3\right).
\end{equation*}
Thus, $\nabla R=0$, i.e. $g_{\varphi}$ is locally symmetric, if and only if $b_{21}=0$.
\end{proof}
Since Theorem \ref{th:nilpotent} gives us only a description of the holonomy of the induced pseudo-Riemannian metric in dependence of the parameters in Equation (\ref{eq:degnilpotent}), we end this section by determining which parameter values correspond to which seven-dimensional nilpotent almost Abelian Lie algebra. Note that a complete list of these Lie algebras is given in Table \ref{table1} and the names occurring in the following proposition are also taken from this table.
\begin{proposition}\label{pro:nilpotent}
Let $\g$ be a seven-dimensional nilpotent almost Abelian Lie algebra and $\uf$ be a codimension one Abelian ideal. Assume that there exists $B=(b_{ij})_{ij}\in \bR^{2\times 2}$, $v=(v_1,v_2)\in \bR^2$, $w=(w_1,w_2)\in \bR^2$, $\delta\in \{-1,0,1\}$ and $f_7\in \g\backslash \uf$ such that $\ad(f_7)|_{\uf}$ is as in Equation (\ref{eq:degnilpotent}) with respect to some basis $(f_1,\ldots,f_6)$ of $\uf$.
\begin{enumerate}
\item
If $\delta\neq 0$, then
\begin{equation*}
\g\in\left\{\mathfrak{n}_{7,3},\mathfrak{n}_{7,4},\mathfrak{n}_{6,1}\oplus \bR,A_{5,1}\oplus \bR^2,A_{5,2}\oplus \bR^2\right\}.
\end{equation*}
More exactly, the following is true:
\begin{itemize}
\item[(i)]
$\g=\mathfrak{n}_{7,4}$ if and only if $v_1\neq 0$.
\item[(ii)]
$\g=\mathfrak{n}_{7,3}$ if and only if $v_1=0$, $b_{21}\neq 0$ and $\tr(B) \neq -\delta v_2^2$.
\item[(iii)]
$\g=A_{5,2}\oplus \bR^2$ if and only if $v_1=0$, $b_{21}\neq 0$ and $\tr(B)=-\delta v_2^2$.
\item[(iv)]
$\g=\mathfrak{n}_{6,1}\oplus \bR$ if and only if $v_1=b_{21}=0$ and either $w_1\neq \delta b_{11} v_2$ or $\tr(B)\neq -\delta v_2^2$.
\item[(v)]
$\g=A_{5,1}\oplus \bR^2$ if and only if $v_1=b_{21}=0$, $w_1= \delta b_{11} v_2$ and $\tr(B)= -\delta v_2^2$.
\end{itemize}
\item
If $\delta=0$, then
\begin{equation*}
\g\in\left\{\mathfrak{n}_{7,1},\mathfrak{n}_{7,2},\mathfrak{n}_{6,1}\oplus \bR,A_{5,1}\oplus \bR^2,\mathfrak{h}_3\oplus \bR^4,\bR^7\right\},
\end{equation*}
and, conversely, all these Lie algebras admit a codimension one Abelian ideal $\uf$, a basis $f_1,\ldots,f_6$ of $\uf$ and $f_7\in \g\backslash \uf$ such that $\ad(f_7)|_{\uf}$ is as in Equation (\ref{eq:degnilpotent}).
\end{enumerate}
\end{proposition}
\begin{proof}
We set $F:=\ad(f_7)|_{\uf}$.
\begin{enumerate}
\item
Let $\delta\in \{-1,1\}$.

By direct computation, we get that the only non-zero entry of $F^5$ is $F^5[1,6]=-v_1^3$. Thus, the Jordan normal form consists of only one Jordan block of size six, i.e. $\g=\mathfrak{n}_{7,4}$, if and only if $v_1\neq 0$.

Let $v_1=0$. Then we get that $F^4=0$ and that the only non-zero elements of $F^3$ are $F^3[1,6]=\delta v_2 b_{21}$ and $F^3[3,6]=b_{21}$. Hence, the Jordan normal form of $F$ has a Jordan block of size four if and only if $b_{21}\neq 0$. To determine under which condition the Jordan normal form of $F$ has Jordan blocks of sizes $(4,2)$, i.e. $\g=\mathfrak{n}_{7,2}$, and when it has Jordan blocks of sizes $(4,1,1)$, i.e. $\g=A_{5,2}\oplus \bR^2$, we have to compute the rank of $F$. Since the first column, the third column and the last row of $F$ are zero and the second row is a multiple of the fifth row, the rank is the same as the rank of the following $4\times 4$-submatrix of $F$:
\begin{equation*}
G:=\begin{pmatrix} 
-\tr(B)  & v_2 & w_1 & w_2 \\
v_2  & \delta & b_{11} & b_{12} \\
0   & 0 & b_{21} & b_{22} \\
 0  & 0 & 0 & \delta
 \end{pmatrix}\in \bR^{4\times 4}
\end{equation*}
Since the rank of $F$ is three or four, it suffices to compute the determinant of $G$. We have $\det(G)=-\delta b_{21} \left(\delta\cdot\tr(B)+v_2^2\right)$ and so $\det(G)\neq 0$ if and only if $\tr(B)\neq -\delta v_2^2$. Hence, (ii) and (iii) follow.

So let us now assume that $v_1=b_{21}=0$. Then the only non-zero elements of $F^2$ are $F^2[1,6]=-b_{11} v_2+\delta w_1$ and $F^2[3,6]=v_2^2+\delta\cdot \tr(B)$. Thus, the Jordan normal form of $F$ has a Jordan block of size three if and only if $w_1\neq \delta b_{11} v_2$ or $\tr(B)\neq -\delta v_2^2$. Again, we have to compute the rank of $F$, which is the same as the rank of $G$. Since the third column of $G$ is a multiple of the last column, we can do this by computing the rank of the matrix 
\begin{equation*}
\begin{pmatrix} 
-\tr(B)  & v_2 & w_1 & w_2 \\
v_2  & \delta & b_{11} & b_{12} \\
 0  & 0 & 0 & \delta
 \end{pmatrix}\in \bR^{3\times 4}.
\end{equation*}
If $w_1\neq \delta b_{11} v_2$, then the last three columns are linearly independent and if $\tr(B)\neq -\delta v_2^2$, then the first, the second and the fourth column are linearly independent. Thus, the rank is always three and so the Jordan normal form has Jordan blocks of sizes $(3,2,1)$, i.e. $\g=\mathfrak{n}_{6,1}\oplus \bR$.

So let us now assume that $v_1=b_{21}=0$, $w_1=\delta b_{11} v_2$ and $\tr(B)=-\delta v_2^2$. Again, we have to compute the rank of $F$ and may do this by computing the rank of
\begin{equation*}
\begin{pmatrix} 
-\tr(B)  & v_2 & w_1 & w_2 \\
v_2  & \delta & b_{11} & b_{12} \\
 0  & 0 & 0 & \delta
 \end{pmatrix}\in \bR^{3\times 4}.
\end{equation*}
Since $\delta \neq 0$, the rank is at least two. From $w_1=\delta b_{11} v_2$, $\tr(B)=-\delta v_2^2$ and $\delta^2=1$, we get that the first and the third column are multiples of the second one. Hence, the rank is two and the Jordan normal form of $F$ has Jordan blocks of sizes $(2,2,1,1)$, i.e. $\g=A_{5,1} \oplus \bR^2$.
\item
Let $\delta=0$. Then $F^3=0$ and one can show that $F^2=0$ if and only if $v=0$.

Suppose first that $v\neq 0$, i.e. $v_1\neq 0$ or $v_2\neq 0$. Then $F$ has one Jordan block of size three and the rank of $F$ is at most four. Moreover, the submatrix of $F$ obtained by erasing the first column and the last two rows is given by
\begin{equation*}
\begin{pmatrix} 
-\tr(B)  & v_1 & v_2 & w_1 & w_2 \\
0  & 0 & 0 & v_1 & v_2 \\
v_2 & 0 & 0 & b_{11} & b_{12} \\
-v_1 & 0 & 0 & b_{21} & b_{22}
 \end{pmatrix}\in \bR^{4\times 5},
\end{equation*}
and has at least rank three since $v_1\neq 0$ or $v_2\neq 0$. If we choose $v=(1,1)$, $w=0$ and  $B=\diag(1,0)$, then $F$ has rank four and so the Jordan normal form of $F$ has Jordan blocks of sizes $(3,3)$, i.e. $\g=\mathfrak{n}_{7,2}$. If we choose $v=(1,1)$, $w=0$ and $B=0$, then $F$ has rank three and so the Jordan normal form of $F$ has Jordan blocks of sizes $(3,2,1)$, i.e. $\g=\mathfrak{n}_{6,1}\oplus \bR$.

Suppose now that $v=0$. Then the rank of $F$ is the same as the rank of the following $3\times 3$-submatrix:
\begin{equation*}
\begin{pmatrix} 
-\tr(B)   & w_1 & w_2 \\
0   & b_{11} & b_{12} \\
0   & b_{21} & b_{22}
 \end{pmatrix}\in \bR^{3\times 3}.
\end{equation*}
Obviously, we may choose the parameters $B=(b_{ij})_{ij}$ and $w_1,\, w_2$ in such a way that the rank of the above matrix is equal to any of the numbers $3,\,2,\,1,\, 0$. Hence, any Jordan normal form with Jordan blocks of maximal size two is possible and the corresponding seven-dimensional nilpotent almost Abelian Lie algebras are $\mathfrak{n}_{7,1}$, $A_{5,1} \oplus \bR^2$, $\mathfrak{h}_3\oplus \bR^4$ and $\bR^7$.
\end{enumerate}
\end{proof}
\section*{Acknowledgments}
The author thanks the University of Hamburg for financial support, Vicente Cort\'es for helpful comments on a draft version of this paper and  Anna Fino and Ignacio Luj\'{a}n for sharing a draft version of their paper \cite{FL} with the author.
\section*{Appendix}

Table \ref{table1} contains a list of all seven-dimensional nilpotent almost Abelian Lie algebras and gives us the information on the existence of parallel $\G_2^*$-structures and on geometric properties of the induced pseudo-Riemannian metrics on these Lie algebras that we get from Theorem \ref{th:nilpotent} and Proposition \ref{pro:nilpotent}. The names for the Lie algebras in the first column are taken from \cite{PSWZ} for the four-dimensional and five-dimensional summands. For the six-dimensional summands and the seven-dimensional indecomposable Lie algebras, we use the class symbol $\mathfrak{n}$, indicate the dimension with the first index and enumerate the ones with the same dimension with the second index. The second column contains the Lie bracket in the well-known dual notation, i.e. we write down $\left(de^1,\ldots,de^7\right)$ for a basis $\left(e^1,\ldots,e^7\right)$ of $\g^*$. The third column indicates whether the Lie algebra $\g$ admits a parallel $\G_2^*$-structure or not. The next column, labeled ``$\dim{Hol}$'' gives the possible dimensions of the holonomy groups of a parallel $\G_2^*$-structure on the Lie algebra $\g$. If there is no such structure on $\g$, we write ``-''. Finally, the last column tells us whether or not there exists a parallel $\G_2^*$-structure on $\g$ which induces non-flat locally symmetric pseudo-Riemannian metric.\\

\small
\setlength{\LTcapwidth}{12.5cm}
\renewcommand{\arraystretch}{1.56} 
\setlength{\tabcolsep}{0.17cm}
\begin{longtable}{lL{3.7cm}C{2cm}C{1.6cm}C{2cm}}
  \caption{parallel $\G_2^*$-structures on 7d nilpotent almost Abelian Lie algebras} \\
  \hline $\g$ & Lie bracket & parallel $\G_2^*$-structure & $\dim{Hol}$  & non-flat loc. symmetric  \\ \hline
  \endfirsthead
  \endhead\label{table1}\\
  $\mathfrak{n}_{7,1}$ & $\left(e^{47},e^{57},e^{67},0,0,0,0\right)$ & yes  & 0 & no \\
  $\mathfrak{n}_{7,2}$ & $\left(e^{27},e^{37},0,e^{57},e^{67},0,0\right)$ & yes  & 0 & no \\
  $\mathfrak{n}_{7,3}$ & $\left(e^{27},e^{37},e^{47},0,e^{67},0,0\right)$ & yes & 2 & no \\
  $\mathfrak{n}_{7,4}$ & $\left(e^{27},e^{37},e^{47},e^{57},e^{67},0,0\right)$ & yes & 0, 1, 2 & yes \\
  $\mathfrak{n}_{6,1}\oplus \bR$ & $\left(0,0,e^{12},e^{13},0,e^{15},0\right)$  & yes & 0, 1 & yes\\
   $\mathfrak{n}_{6,2}\oplus \bR$ & $\left(0,0,e^{12},e^{13},e^{14},e^{15},0\right)$ & no & - & no \\
    $A_{5,1}\oplus \bR^2$ & $(\e^{35}, \e^{45}, 0, 0, 0,0,0)$ & yes & 0, 1 & yes \\
   $A_{5,2}\oplus \bR^2$ & $(\e^{25}, \e^{35}, \e^{45}, 0, 0,0,0)$ & yes & 2  & no \\
   $A_{4,1}\oplus \bR^3$ & $(\e^{24}, \e^{34}, 0, 0,0,0,0)$ & no & - & no  \\
   $\mathfrak{h}_3\oplus \bR^4$ & $(e^{23},0,0,0,0,0,0)$ & yes & 0 & no \\
   $\bR^7$ & $(0,0,0,0,0,0,0)$ & yes & 0 & no
\end{longtable}

\end{document}